\documentclass{amsart}

\usepackage[hmargin=3.5cm,top=3cm,bottom=3cm,includehead]{geometry}

\usepackage{amsmath, amssymb, amsthm, mathrsfs}
\usepackage{graphicx,color,enumerate,microtype} 

\usepackage[pdfborder={0 0 0}]{hyperref}

\newcommand\N{{\mathbb N}}
\newcommand\R{{\mathbb R}}

\def\AA{{\mathcal A}}
\def\BB{{\mathcal B}}

\def\DD{{\mathcal D}}
\def\EE{{\mathcal E}}

\def\LL{{\mathcal L}}

\def\NN{{\mathcal N}}

\def\SS{{\mathcal S}}

\def\BBB{{\mathscr B}}

\def\eps{{\varepsilon}}

\newcommand{\la}{\langle}
\newcommand{\ra}{\rangle}

\newcommand{\Indiq}[1]{\mathbf{1}_{\{ #1 \}}}

\newtheorem{thm}{Theorem}[section]
\newtheorem{prop}[thm]{Proposition}
\newtheorem{lem}[thm]{Lemma}
\newtheorem{cor}[thm]{Corollary}
\newtheorem*{thm*}{Theorem}
\theoremstyle{remark}
\newtheorem{rem}[thm]{Remark}

\theoremstyle{definition}
\newtheorem{definition}[thm]{Definition}

\numberwithin{equation}{section}

\newcommand{\beqn}{\begin{equation}}
\newcommand{\eeqn}{\end{equation}}
\newcommand{\bal}{\begin{aligned}}
\newcommand{\eal}{\end{aligned}}

\title[Convergence to equilibrium for the Landau equation]
{On the rate of convergence to equilibrium for the homogeneous Landau equation with soft potentials}

\author{Kleber Carrapatoso}   



\begin{document}

\maketitle

\begin{center}
\small
\textit{
\'Ecole Normale Sup\'erieure de Cachan, CMLA (UMR CNRS 8536), 61 av.\ du Pr\'esident Wilson, 94235~Cachan, France. E-mail address}: \texttt{carrapatoso@cmla.ens-cachan.fr}
\end{center}

\begin{abstract}
We investigate in this work the rate of convergence to equilibrium of solutions to the spatially homogeneous Landau equation with soft potentials. 
Firstly, we prove a polynomial in time convergence using an entropy method with some new a priori estimates. Finally, we prove
an exponential in time convergence towards the equilibrium with the optimal rate, given by the spectral gap of the associated linearised operator, combining new decay estimates for the semigroup generated by the linearised Landau operator in weighted $L^p$-spaces together with the polynomial decay described above.

%
%
%
\end{abstract}

\section{Introduction}

The Landau equation is a fundamental model in kinetic theory that describes the evolution of the density of particles in a plasma in the phase space of all positions and velocities.
We consider in this work the case of spatially homogeneous density functions, which verifies the \emph{spatially homogeneous Landau equation} given by
\beqn\label{eq:landau}
\left\{
\begin{array}{rcl}
\partial_t f  &=&  Q(f,f)  \\
f_{|t=0} &=& f_0,
\end{array}
\right.
\eeqn
where $f=f(t,v) \geq 0$ is the density of particles with velocity $v \in \R^3$ at time $t \ge 0$. The Landau collision operator $Q$ is a bilinear operator acting only on the variable $v$ and given by
\beqn\label{eq:oplandau0}
Q(g,f) = \partial_{i} \int_{\R^3} a_{ij}(v-v_*) \left[ g_* \partial_j f - f \partial_{j}g_*\right]\, dv_*,
\eeqn
where here and below we shall use the convention of implicit summation over repeated indices and the usual shorthand $g_* = g(v_*)$, $\partial_{j} g_* = \partial_{v_{*j}} g(v_*)$, $f=f(v)$ and $\partial_j f = \partial_{v_j} f(v)$. 

The matrix-valued function $a$ is nonnegative, symmetric and depends on the interaction between particles. One usually assumes that particles interact by binary relation through a potential proportional to $1/r^s$, where $r$ denotes their distance. In this case $a$ is given by (see for instance \cite{Villani-BoltzmannBook})
\beqn\label{eq:aij}
a_{ij}(z) = |z|^{\gamma+2} \, \Pi_{ij}(z), \quad \Pi_{ij}(z) = \left( \delta_{ij} - \frac{z_i z_j}{|z|^2}\right),
\eeqn
with $\gamma=(s-4)/s$. One usually calls hard potentials if $\gamma\in(0,1]$, Maxwellian molecules if $\gamma=0$, soft potentials if $\gamma=(-3,0)$ and Coulombian potential if $\gamma=-3$. One also separates the soft potentials into two categories: \emph{moderately} soft potentials when $\gamma \in (-2,0)$ and \emph{very} soft potentials if $\gamma \in (-3,-2]$. In this paper we are interested in the case of moderately soft potentials.

We also define the following quantities
\beqn\label{eq:bc}
b_i(z) = \partial_j a_{ij}(z) = - 2 \, |z|^\gamma \, z_i, \quad c(z) =  \partial_{ij} a_{ij}(z)  = - 2 (\gamma+3) \, |z|^\gamma,
\eeqn
from which we are able to rewrite the Landau operator in the following way
\beqn\label{eq:oplandau}
\bal
Q(g,f) 
& = \nabla \cdot \{ (a*g) \nabla f - (b*g) f \}\\
&= ( a_{ij}*g)\partial_{ij} f - (c*g) f .
\eal
\eeqn

\smallskip

Let us present some important properties of the Landau equation. First of all, it conserves mass, momentum and energy. Indeed, at least formally, for any test function $\varphi$ we have (see e.g. \cite{Vi2})
$$
\int_{\R^3} Q(f,f) \varphi(v) \, dv = \frac12 \int_{\R^3 \times \R^3} a_{ij}(v-v_*) f f_* 
\left(\frac{\partial_i f}{f} -  \frac{\partial_{i} f_*}{f_*}  \right)\left( \partial_j \varphi - \partial_j \varphi_* \right) \, dv \, dv_*,
$$
from which we deduce, for any $t\ge 0$,
\beqn\label{eq:cons}
\frac{d}{dt} \int f \varphi\, dv = 
\int Q(f,f) \varphi \, dv = 0 \qquad\text{for}\qquad \varphi(v)  = 1, v, |v|^2.
\eeqn
Another important property of this equation is the Landau version of the celebrated $H$-Theorem of Boltzmann: The entropy $H(f) :=\int f \log f$ is nonincreasing and any equilibrium is a Maxwellian distribution (Gaussian distribution). Indeed, at least formally, the entropy-dissipation functional defined as
\beqn\label{eq:def-Df}
D(f) := -\int Q(f,f) \log f,
\eeqn
verifies the following inequality
\begin{equation}\label{eq:Df}
\begin{aligned}
D(f) = -\frac{d}{dt} H(f)=
\frac{1}{2}\int_{\R^3\times \R^3} \,a_{ij}(v-v_*)
\left( \frac{\partial_i f}{f} -  \frac{\partial_{i*} f_*}{f_*} \right)
 \left( \frac{\partial_j f}{f} -  \frac{\partial_{j*} f_*}{f_*} \right)\,  f f_* \,dv \,dv_*
\geq 0,
\end{aligned}
\end{equation}
and we also have
\beqn\label{eq:Ht}
H(f(t)) + \int_0^t D(f(\tau)) \, d\tau = H(f_0).
\eeqn
From this, it also follows that any equilibrium is a Maxwellian distribution
$$
\mu_{\rho,u,T} (v) := \frac{\rho}{(2\pi T)^{3/2}} e^{-\frac{|v-u|^2}{2T}},
$$
for some $\rho >0$, $u\in \R^3$ and $T>0$. 

It is then expected that any solution $f(t,\cdot)$ converges towards the Maxwellian equilibrium $\mu_{\rho_f, u_f, T_f}$ when $t\to +\infty$, where $\rho_f$ is the density of the gas, $u_f$ the mean velocity and $T_f$ the temperature, defined by
$$
\rho_f = \int f(v), \quad
u_f = \frac{1}{\rho} \int v f(v),\quad
T_f = \frac{1}{3\rho}\int |v-u|^2 f(v),
$$
and these quantities are defined by the initial datum $f_0$ thanks to the conservation properties of the Landau operator \eqref{eq:cons}. 

\smallskip

We shall always assume that $f_0$ is a nonnegative function with finite mass, energy and entropy, more precisely
$$
\int f_0 = M_0 < \infty, \quad \int |v|^2 f_0 = E_0 < \infty, \quad \int f_0 \log f_0 = H_0 < \infty,
$$
and it is classical that this implies
\beqn\label{eq:f0}
f_0 \in L^1_2 \cap L \log L, \qquad L \log L := \left\{ f \in L^1 \mid  \int |f| \log ( |f|) < \infty  \right\}. 
\eeqn
Furthermore, we may only consider the case of initial datum $f_0$ satisfying
\beqn\label{f0}
f_0 \in L^1_{1,0,1} := \{ f \in L^1 \mid \rho_f=1, \; u_f = 0, \; T_f=1  \},
\eeqn
the general case being reduced to \eqref{f0} by a simple change of coordinates.
We shall then denote $\mu(v) = (2\pi)^{-3/2} e^{-|v|^2/2}$ the standard Gaussian distribution in $\R^3$, which corresponds to the Maxwellian with same mass, momentum and energy of $f_0$.

\smallskip

We can linearise the Landau equation around the equilibrium $\mu$, with the perturbation $ f(t,v)=\mu (v) + h(t,v)$, which satisfies at the first order the \emph{linearised Landau equation}
\beqn\label{eq:lin}
\left\{
\begin{array}{rcl}
\partial_t h  &=&  \LL h  \\
h_{|t=0} &=& h_0,
\end{array}
\right.
\eeqn
where the initial datum is defined by $h_0 =f_0 - \mu$, and 
where the linearised Landau operator $\LL$ is given by
\beqn\label{eq:oplandaulin}
\LL h= Q(\mu,h)+Q(h,\mu).
\eeqn
Furthermore, from the conservation properties \eqref{eq:cons}, we observe that the null space of $\LL$ has dimension $5$ and is given by (see e.g. \cite{DL,Guo,BM,M,MS})
\beqn\label{eq:kerLL}
\NN(\LL)  = \mathrm{Span} \{\mu, v_1\mu,v_2 \mu, v_3\mu, |v|^2 \mu\}.
\eeqn
Consider the weighted Hilbert space $L^2(\mu^{-1/2})$ associated with the following scalar product and norm
$$
\la h , g \ra_{L^2(\mu^{-1/2})} := \int h g \, \mu^{-1}
\quad\text{and}\quad
\| h \|_{L^2(\mu^{-1/2})}^2 := \int |h|^2\, \mu^{-1}.
$$
A simple computation gives
$$
\bal
&\la \LL h , h \ra_{L^2(\mu^{-1/2})} \\
&\qquad
= -\frac12 \iint a_{ij}(v-v_*) \{ \partial_i(\mu^{-1} h) - \partial_{*i}(\mu^{-1}_* h_*)  \} \{ \partial_j(\mu^{-1} h) - \partial_{*j}(\mu^{-1}_* h_*)  \}
\, \mu_* \, \mu \, dv_* \, dv \\
&\qquad \le 0,
\eal
$$
which implies that $\LL$ is self-adjoint on $L^2(\mu^{-1/2})$ and, moreover, that the spectrum of $\LL$ in $L^2(\mu^{-1})$ is included in $\R_-$.

\subsection{Existing results}
Let us mention known results concerning the long-time behaviour of solutions to the Landau equation (and for a more detailed presentation we refer to \cite{KC4}).

In the Maxwellian molecules case $\gamma=0$, Villani~\cite{Vi1} proves an exponential in time convergence to equilibrium. For hard potentials $\gamma \in (0,1]$, Desvillettes and 
Villani~\cite{DesVi2} obtain a polynomial in time convergence to equilibrium, and more recently we prove in \cite{KC4} an optimal exponential decay to equilibrium. Moreover, Toscani and Villani~\cite{TosVi-slow} also prove a decay to equilibrium polynomially in time, in the case of \emph{mollified} soft potentials $\gamma \in (-3,0)$, which corresponds to replace $|z|^{\gamma + 2}$ in \eqref{eq:aij} by a mollified function $\Psi(z)$ truncating the singularity at the origin (see Section~\ref{ssec:poly} for more details). 
It is worth mentioning that all the results from \cite{Vi1,DesVi2,TosVi-slow} above are purely nonlinear and based on an entropy method.

\smallskip

Another approach for studying the long-time behaviour consists in considering the linearised equation around the equilibrium \eqref{eq:lin}, which has been investigated by several authors. Summarising results of Degond and Lemou~\cite{DL}, Guo~\cite{Guo}, Baranger and Mouhot~\cite{BM}, Mouhot~\cite{M}, Mouhot and Strain~\cite{MS}, we have the following proposition:

\begin{prop}\label{prop:gap}
Let $\gamma \in [-2,1]$. There exists a constructive constant $\lambda_0 >0$ (spectral gap) such that, for any $h \in L^2(\mu^{-1/2})$ with $h \in \NN(\LL)^{\perp}$,
$$
\la \LL h , h \ra_{L^2(\mu^{-1/2})} \le  - \lambda_0 \| h \|_{L^2(\mu^{-1/2})}^2.
$$ 
As a consequence we obtain an exponential decay for the linearised Landau equation \eqref{eq:lin}: for any $t \ge 0$ and $h \in L^2(\mu^{-1/2})$, there holds
$$
\| e^{t \LL} h - \Pi_0 h \|_{L^2(\mu^{-1/2})} \le e^{- \lambda_0 t} \| h - \Pi_0 h\|_{L^2(\mu^{-1})},
$$
where $\Pi_0$ is the projection onto $\NN(\LL)$.

\end{prop}

\subsection{Main results and strategy}
Let us define the notion of weak solution we consider in this paper.

\begin{definition}[Weak solutions \cite{Vi2}]
Let $\gamma \in [-2,1]$ and consider a nonnegative $f_0 \in L^1_2 \cap L \log L$. 
We say that $f$ is a weak solution of the Cauchy problem \eqref{eq:landau} if the following conditions are fulfilled:

\begin{enumerate}[($i$)]

\item $f\ge 0$, $f \in C ([0,\infty); \DD') \cap L^\infty([0,\infty); L^1_2 \cap L \log L) \cap L^1_{loc}([0,\infty); L^1_{2+\gamma})$;

\item $f(0) = f_0$;

\item for any $t\ge 0$
$$
\int f(t) \varphi = \int f_0 \varphi \quad\text{for}\quad \varphi(v) = 1, v , |v|^2;
\quad\text{and}\quad
H(f(t)) + \int_0^t D(f(\tau)) \le H(f_0);
$$

\item $f$ verifies \eqref{eq:landau} in the distributional sense: for any $\varphi \in C([0,\infty); C^\infty_c)$, for any $t \ge 0$,
$$
\int f(t) \varphi(t) - \int f_0 \varphi(0) - \int_0^t \int f(\tau) \partial_t \varphi(\tau)
= \int_0^t \int Q(f(\tau), f(\tau)) \varphi(\tau),
$$ 
where the last integral in the right-hand side is defined by 
$$
\int Q(f,f) \varphi = \frac12 \iint a_{ij}(v-v_*) (\partial_{ij} \varphi + \partial_{ij}\varphi_*) \, f_* f 
+ \iint b_{i}(v-v_*) (\partial_{i} \varphi - \partial_{i}\varphi_*) \, f_* f .
$$

\end{enumerate}
It is observed in \cite{Vi2} that these formulae make sense as soon as $f$ satisfies ($i$) and $\varphi \in W^{2,\infty}(\R^3)$.

\end{definition}

In the case of moderately soft potentials $\gamma \in (-2,0)$, 
it is proven in \cite{Vi2} that if $f_0 \in L^1_2 \cap L \log L$ there exists a global weak solution. If moreover we assume $f_0 \in L^1_{k}$, with $k > \gamma^2/(2+\gamma)$, then the weak solution is unique \cite[Corollary 4]{FG}.

\medskip

We can now state our main results on the rate of convergence to equilibrium: a polynomial convergence in Theorem~\ref{thm:main0} and then an exponential convergence in Theorem~\ref{thm:main}.

\begin{thm}[Polynomial convergence]\label{thm:main0}
Let $\gamma \in (-2,0)$ and $ f_0 \in L^1_{k+8-3\gamma/4} \cap L \log L$ with $k>7|\gamma|/2$. Then there exists a weak solution $f$ to the Landau equation associated to $f_0$ such that
$$
\forall\, t\ge 0, \qquad
H( f(t) | \mu ) \le C (1+t)^{-\frac{k}{|\gamma|} + \frac72},
$$
for some constructive constant $C>0$ and where $H(f | \mu) := \int f \log (f / \mu)$ is the relative entropy of $f$ with respect to $\mu$.

\end{thm}

The proof of Theorem~\ref{thm:main0} follows the strategy introduced by Toscani and Villani~\cite{TosVi-slow} (see Section~\ref{ssec:poly} for more details), in which, as already explained, a polynomial in time convergence to equilibrium for mollified soft potentials is proven. This strategy was developed in order to treat the trend to equilibrium issue for kinetic equations with relatively bad control of the distribution tails (as for Boltzmann and Landau-type equations with soft potentials) and they compensate the lack of uniform in time estimates by some precise logarithmic Sobolev inequalities. In order to use this strategy, we prove some new \emph{a priori} estimates for the evolution of weighted $L^1$ and Sobolev norms in Section~\ref{sec:apriori}. Then we prove Theorem~\ref{thm:main0} in Section~\ref{ssec:poly} using these \emph{a priori} estimates together with a functional inequality relying entropy and entropy-dissipation from \cite{TosVi-slow}.

\medskip

\begin{thm}[Exponential convergence]\label{thm:main}
Let $\gamma \in (-1,0)$ and $ f_0 \in L \log L \cap L^1(e^{\kappa \la v \ra^s})$ with $\kappa>0$ and $- \gamma < s< 2 + \gamma$. Then the unique weak solution $f$ to the Landau equation associated to $f_0$ satisfies
$$
\forall\, t\ge 0, \qquad
\| f(t) - \mu \|_{L^1} \le C e^{-\lambda_0 t},
$$
for some constructive constant $C>0$ and where $\lambda_0>0$ is the spectral gap of the associated linearised operator. 

\end{thm}

\begin{rem}
The restriction $\gamma \in (-1,0)$ comes from the fact that we need $s + \gamma >0$ in order to prove the "spectral gap/semigroup decay" extension theorem for the linearised equation (see Theorem~\ref{thm:trou}) and $s < \gamma + 2$ to prove the propagation of stretched exponential moments (see Lemma~\ref{lem:momentsexp}).
\end{rem}

The strategy to prove this theorem is based on:

\begin{enumerate}[(1)]

\item New exponential decay estimates (with sharp rate) for the semigroup generated by the linearised Landau operator $\LL$ in various $L^p$-spaces with stretched exponential weight, using a method developed in \cite{GMM}. This question is addressed in Section~\ref{sec:linear}.

\item New \emph{a priori} estimates for the nonlinear equation proved in Section~\ref{sec:apriori} and the convergence to equilibrium from Theorem~\ref{thm:main0} proven in Section~\ref{ssec:poly}.

\item A ``coupling method" in order to connect the linearised theory with the nonlinear one: for small times we use the polynomial convergence from Theorem~\ref{thm:main0}; then for large times we use (2) to prove that the solution enters in a suitable neighbourhood of the equilibrium, in which the linear part is dominant, and we have an optimal exponential decay from (1). This is proven in Section~\ref{ssec:exp}.

\end{enumerate}

It is worth mentioning that this strategy has been used by several authors and for different equations in order to prove an exponential in time convergence to equilibrium. It was first introduced by Mouhot~\cite{Mouhot2} for the homogeneous Boltzmann equation for hard potentials with Grad's cut-off. This same approach was later used by Gualdani, Mischler and 
Mouhot~\cite{GMM} for the inhomogeneous Boltzmann equation for hard spheres on the torus and for the Fokker-Planck equation, and also by Mischler and Mouhot \cite{MiMo-FP} for Fokker-Planck equations. More recently, the author~\cite{KC4} used it for the homogeneous Landau equation with hard potentials, and Tristani~\cite{Tristani} for the homogeneous Boltzmann equation for hard potentials without cut-off.

\subsection{Notations}
Let $m:\R^d \to \R^+$ be a weight function. For any $1 \le p \le \infty $ we define the weighted space $L^p(m)$ associated with the norm
$$
\| f \|_{L^p(m)} := \| m f \|_{L^p}.
$$
We also define higher-order weighted Sobolev spaces $W^{\ell,p}(m)$ associated with the norm
$$
\| f \|_{W^{\ell,p}(m)}^p := \sum_{|\alpha|\le \ell} \| \partial^\alpha f \|_{L^p(m)}^p, \qquad 1 \le p < \infty,
$$
with the usual modification for $p=\infty$ and for homogeneous spaces $\dot W^{\ell,p}(m)$.
When $m = \la v \ra^k := (1 + |v|^2)^{k/2}$ is a polynomial weight, we denote $W^{\ell,p}_k := W^{\ell,p}( \la v \ra^k)$.

\smallskip

Let $X,Y$ be Banach spaces and consider a linear operator $\Lambda : X \to Y$. We shall denote by $\SS_{\Lambda}(t) = e^{t\Lambda}$ the semigroup generated by $\Lambda$. Moreover we denote by $\BBB(X,Y)$ the space of bounded linear operators from $X$ to $Y$ and by $\| \cdot \|_{\BBB(X,Y)}$ its norm operator, with the usual simplification $\BBB(X) = \BBB(X,X)$.

\section{The linearised operator}\label{sec:linear}

In this section we shall denote
\beqn\label{eq:barabc}
\bar a_{ij}(v) = a_{ij}*\mu, \quad
\bar b_i(v) = b_i * \mu, \quad
\bar c(v) = c*\mu.
\eeqn

Let us now make our assumptions on the weight function $m=m(v)$:

\medskip

\noindent{\bf(W) Stretched exponential weight.} 
We consider a weight function $m = \exp(\kappa \la v \ra^s)$ with $\kappa>0$, $0<s< 2$ and $s+\gamma>0$.

\medskip

We are now able to state the main result of this section, which extends to various weighted $L^p$ spaces the decay of the semigroup $\SS_\LL(t)$ generated by the operator $\LL$, known to hold in $L^2(\mu^{-1/2})$ by Proposition~\ref{prop:gap}.

\begin{thm}\label{thm:trou}
Let $\gamma \in (-2,0)$, $1 \le p \le 2$ and a weight function $m$ satisfying (W). Then there exists a constant $C>0$ such that, for all $t \ge 0$ and any $h \in L^p(m)$, there holds
$$
\| \SS_\LL(t) h - \Pi_0 h \|_{L^p(m)} \le C e^{-\lambda_0 t} \| h - \Pi_0 h \|_{L^p(m)},
$$
where $\Pi_0$ is the projection onto $\NN(\LL)$ and $\lambda_0 >0$ is the spectral gap of $\LL$ on $L^2(\mu^{-1/2})$.

\end{thm}

In order to prove this theorem we shall use the method of \emph{enlargement of the functional space of semigroup decay} developed by Gualdani, Mischler and Mouhot \cite{GMM}. Roughly speaking, if one knows some quantitative information on the semigroup decay associated with an operator $\LL$ in some \emph{small} space $E$, this method enables one to deduce this quantitative estimate on a \emph{larger} space $\EE \supset E$, when the operator $\LL$ satisfies some properties. In order to do that, we need to factorise $\LL = \AA + \BB$ and to prove some properties for these operators, namely that $\BB$ has a well localised spectrum (see Section~\ref{ssec:hypo}) and $\AA$ is regularising in some sense (see Section~\ref{ssec:regularizing}).

\subsection{Factorisation of the operator}

Using the form \eqref{eq:oplandau} of the operator $Q$, we decompose the linearised Landau operator $\LL$ defined 
in \eqref{eq:oplandaulin} as $\LL =  \AA_0 + \BB_0$, where we define
\beqn\label{eq:A0B0}
\bal
\AA_0 h &:= Q(h,\mu) = (a_{ij}* h)\partial_{ij}\mu - (c * h)\mu, \\
\BB_0 h &:= Q(\mu,h) =(a_{ij}* \mu)\partial_{ij}h - (c * \mu)h.
\eal
\eeqn
Consider a smooth nonnegative function $\chi \in C^\infty_c(\R^3)$ such that $0\leq \chi(v) \leq 1$, $\chi(v) \equiv 1$ for $|v|\leq 1$ and $\chi(v) \equiv 0$ for $|v|>2$. For any $R\geq 1$ we define $\chi_R(v) := \chi(R^{-1}v)$ and in the sequel we shall consider the function $M\chi_R$, for some constant $M>0$.
Then, we make the final decomposition of the operator $\LL$ as $\LL = \AA + \BB$ with
\beqn\label{eq:AB}
\AA  := \AA_0 + M \chi_R ,
\qquad
\BB  := \BB_0 - M\chi_R,
\eeqn
where $M$ and $R$ will be chosen later.

\subsection{Dissipativity properties}
\label{ssec:hypo}

We investigate in this section dissipativity properties of the operator $\BB$.
 
\medskip 
 
First of all, we state the following results concerning $\bar a_{ij}(v)$ (see \cite[Propositions 2.3 and 2.4, Corollary 2.5]{DL} and \cite[Lemma 3]{Guo}) that will be useful.
 
\begin{lem}\label{lem:bar-aij}
The following properties hold:

\begin{enumerate}[(a)]

\item The matrix $\bar a(v)$ has a simple eigenvalue $\ell_1(v)>0$ associated with the eigenvector $v$ and a double eigenvalue $\ell_2(v)>0$ associated with the eigenspace $v^{\perp}$. Moreover,
$$
\bal
\ell_1(v) &= \int_{\R^3} \left(1 - \left(\frac{v}{|v|}\cdot\frac{w}{|w|}   \right)^2   \right) |w|^{\gamma+2} \mu(v-w)\, dw\\
\ell_2(v) &= \int_{\R^3} \left(1 - \frac12 \left| \frac{v}{|v|}\times\frac{w}{|w|}  \right|^2   \right) |w|^{\gamma+2} \mu(v-w)\, dw .
\eal
$$
When $|v|\to +\infty$ we have
$$
\bal
\ell_1(v) &\sim  2 \la v \ra^\gamma \\
\ell_2(v) &\sim  \la v \ra^{\gamma+2} .
\eal
$$
If $\gamma\in (0,1]$ there exists $\ell_0 >0$ such that, for all $v\in\R^3$, 
$\min\{ \ell_1(v), \ell_2(v) \} \geq \ell_0$.

\item The function $\bar a_{ij}$ is smooth, for any multi-index $\beta\in \N^3$
$$
|\partial^\beta \bar a_{ij}(v)| 
\leq C_\beta \la v \ra^{\gamma+2-|\beta|}
$$
and
$$
\bal
\bar a_{ij}(v) \xi_i \xi_j &= \ell_1(v) |P_v \xi|^2 + \ell_2(v)|(I-P_v)\xi|^2, \\
\bar a_{ij}(v) v_i v_j  &= \ell_1(v) |v|^2 ,
\eal
$$
where $P_v$ is the projection on $v$, i.e.
$$
P_v \xi = \left( \xi \cdot \frac{v}{|v|} \right) \frac{v}{|v|}.
$$

\item We have
$$
\bar a_{ii}(v) 
= 2 \int_{\R^3} |v-v_*|^{\gamma+2} \mu(v_*)\, dv_*
\qquad\text{and}\qquad
\bar b_i(v) = - \ell_1(v)\, v_i.
$$

\end{enumerate}

\end{lem}

\bigskip

Let us we define 
\beqn\label{eq:def-phi}
\varphi_{m,p}(v) := \frac1m \, \bar a : \nabla^2 m + \frac{(p-1)}{m^2}\, \bar a \nabla m \cdot \nabla m + \frac2m\, \bar b \cdot \nabla m + (1/p - 1) \bar c .
\eeqn

\medskip

Before proving the desired result in Lemma~\ref{lem:hypoBB}, we give the following elementary lemma to be used in the sequel.

\begin{lem}\label{lem:Jalpha}
Let $J_\alpha(v) := \int_{\R^3} |v-v_*|^\alpha \mu(v_*)\, dv_*$ for $-3 <  \alpha \le 2$. Then it holds:
\begin{enumerate}[(i)]

\item If $0 \le \alpha \le 2$ then $J_\alpha (v) \le |v|^{\alpha} + C_{\alpha}$ for some constant $C_\alpha >0$.

\item If $-3 < \alpha <0$ then $J_\alpha(v) \le C \la v \ra^\alpha $ for some constant $C>0$.

\end{enumerate}

\end{lem}

\begin{proof}
Point $(i)$ can be found in \cite[Lemma 2.5]{KC4}. For point $(ii)$ we observe that the result easily follows if $|v| \le 1$. On the other hand if $|v| >1$ we write
$$
\bal
J_{\alpha}(v) &= \int_{|v_*|\le 1} |v_*|^{\alpha} \mu(v-v_*)\, dv_* + \int_{|v_*|\ge 1} |v_*|^{\alpha} \mu(v-v_*)\, dv_* \\
&\leq \sup_{|v_*|\le 1} \mu(v-v_*) \int_{|v_*| \le 1} |v_*|^{\alpha} \, dv_* 
+ C \int_{|v_*|\ge 1} \la v_* \ra^{\alpha} \mu(v-v_*)\, dv_*.
\eal
$$
Using that $\sup_{|v_*|\le 1} \mu(v-v_*) \le C e^{-|v|^2/4} \le C $ and that $\la v_* \ra^{\alpha} \le C \la v \ra^{\alpha} \la v - v_* \ra^{|\alpha|}$ by Peetre's inequality we conclude to
$$
\bal
J_{\alpha}(v) 
&\leq C
+ C \la v \ra^{\alpha} \int \la v-v_* \ra^{|\alpha|} \mu(v-v_*)\, dv_*
\le C \la v \ra^{\alpha}.
\eal
$$
\end{proof}

\begin{lem}\label{lem:phi}
Let $m$ satisfy assumption (W). Then for all $\lambda > 0$ we can choose $M$ and $R$ large enough such that, for all $v\in\R^3$,
$$
\varphi_{m,p}(v) - M \chi_R (v) \leq -\lambda.
$$
\end{lem}

\begin{proof}
Let $m = \exp(\kappa\la v \ra^s)$. We easily compute
$$
\frac{\nabla m}{m} = \kappa s v \la v \ra^{s-2} 
$$
and
$$
\frac{ (\nabla^2 m)_{ij} }{m} = \kappa s \la v \ra^{s-2} \delta_{ij} + \kappa s(s-2) v_i v_j \la v \ra^{s-4} + \kappa^2 s^2 v_i v_j \la v \ra^{2s-4}.
$$
It follows then
$$
\bal
 \bar a : \frac{\nabla^2 m }{m}
&= (\delta_{ij} \bar a_{ij}) \kappa s \la v \ra^{s-2} + (\bar a_{ij} v_i v_j) \kappa s(s-2) \la v \ra^{s-4}
+ (\bar a_{ij} v_i v_j) \kappa^2 s^2 \la v \ra^{2s-4} \\
&= 2 \kappa s J_{\gamma+2}(v) \la v \ra^{s-2} + \kappa s(s-2) \ell_1(v) |v|^2 \la v\ra^{s-4}
+ \kappa^2 s^2 \ell_1(v) |v|^2 \la v \ra^{2s-4},
\eal
$$
where we have used Lemma~\ref{lem:bar-aij}. Moreover, using again Lemma~\ref{lem:bar-aij}, we obtain
$$
 \bar a \, \frac{\nabla m}{m} \, \frac{\nabla m}{m} = \bar a v v \kappa^2 s^2 \la v \ra^{2s-4} = \kappa^2 s^2 \ell_1(v) |v|^2 \la v\ra^{2s-4}
$$
and
$$
 \bar b \cdot \frac{\nabla m}{m} = - \kappa s \ell_1(v) |v|^2 \la v \ra^{s-2}.
$$
Putting together the above estimates, we obtain
\beqn\label{eq:phi0}
\bal
\varphi_{m,p}(v) 
&= 2 \kappa s J_{\gamma+2}(v) \la v \ra^{s-2} 
+ \kappa s(s-2) \ell_1(v) |v|^2 \la v\ra^{s-4}
+ p \kappa^2 s^2 \ell_1(v) |v|^2 \la v \ra^{2s-4} \\
&\quad
- 2 \kappa s \ell_1(v) |v|^2 \la v \ra^{s-2}
+ 2(\gamma+3)(1-1/p) J_{\gamma}(v).
\eal
\eeqn
From the asymptotic behaviour of $\ell_1$, $J_{\gamma+2}$ and $J_\gamma$, the dominant terms of $\varphi_{m,p}$ in \eqref{eq:phi0} when $|v|\to\infty$ are the first and the fourth one, both of order $\la v \ra^{\gamma+s}$. Using Lemma~\ref{lem:Jalpha} to bound $J_{\gamma+2}(v) \leq \tilde J_{\gamma+2}(v) \underset{|v|\to\infty}{\sim} \la v \ra^{\gamma+2}$ and $\ell_1(v) \underset{|v|\to\infty}{\sim} 2 \la v \ra^{\gamma}$ from Lemma~\ref{lem:bar-aij}, we obtain that
$\varphi_{m,p}(v) \leq \tilde \varphi_{m,p}(v)$ with
\beqn\label{eq:phi-exp}
\tilde \varphi_{m,p}(v) \underset{|v|\to\infty}{\sim} -2 \kappa s \la v \ra^{s+\gamma} \xrightarrow[|v|\to\infty]{} -\infty,
\eeqn
because $s+\gamma >0$ from assumption (W).

Let us fix $\lambda >0$. Then, thanks to \eqref{eq:phi-exp}, we can choose $R$ large enough such that
$$
\forall\, |v| > R, \quad \varphi_{m,p}(v) - M \chi_R(v) \le - \lambda.
$$
Finally, we choose $M \ge \sup_{|v|\le R} \varphi_{m,p}(v) + \lambda$ so that
$$
\forall\, |v| \le R, \quad \varphi_{m,p}(v) - M \chi_R(v) = \varphi_{m,p}(v) - M \le - \lambda,
$$
from which we conclude.
\end{proof}

With the help of the result above, we are able to state a result on the dissipativity of $\BB$. Recall that
$$
\BB = \BB_0 - M\chi_R,
\qquad
\BB_0 f = \nabla\cdot \{ \bar a \nabla f - \bar b f  \}.
$$

\begin{lem}\label{lem:hypoBB}
Let $\gamma\in(-2,0)$, $p\in[1,+\infty)$ and $m$ be a weight function satisfying assumption (W). Then for any $ \lambda > 0 $, we can choose $M$ and $R$ large enough such that the operator $(\BB + \lambda  )$ is dissipative in $L^p(m)$.
\end{lem}

\begin{lem}\label{lem:hypoBBmu}
Let $\gamma\in(-2,0)$. Then for any $ \lambda > 0 $, we can choose $M$ and $R$ large enough such that the operator $(\BB + \lambda  )$ is dissipative in $L^2(\mu^{-1/2})$.
\end{lem}

\begin{proof}[Proof of Lemma \ref{lem:hypoBB}]
We denote $\Phi'(x) = |x|^{p-1}\mathrm{sign}(x)$ and consider the equation
$$
\partial_t h =\BB h= \BB_0 h - M \chi_R h.
$$
For all $p\in[1,\infty)$, we compute
$$
\frac1p \frac{d}{d t} {\| h \|}_{L^p(m)}^p = \int (\BB_0 h) \Phi'(h) m^p - \int (M\chi_R) |h|^p m^p.
$$
For the first term, we perform integration by parts to obtain
$$
\bal
\int (\BB_0 h) \Phi'(h) m^p 
&= \int \nabla\cdot \{ \bar a \nabla h - \bar b h  \} \Phi'(h) m^p \\
&= - \int \bar a \nabla h \nabla(\Phi'(h)) m^p 
- \int \bar a \nabla h \Phi'(h) \nabla(m^p) \\
&\quad
+\int \bar b h \nabla(\Phi'(h)) m^p  
+ \int \bar b h \Phi'(h) \nabla(m^p).
\eal
$$
Using that $\nabla(\Phi'(h)) = (p-1)|h|^{p-2} \nabla h $, $\Phi'(h)\nabla h = p^{-1} \nabla(|h|^p)$ and $h \nabla(\Phi'(h)) = (1-1/p) \nabla(|h|^p)$, and integrating by parts, we finally get
$$
\bal
\int (\BB_0 h) \Phi'(h) m^p 
&= -(p-1) \int \bar a \nabla h \nabla h |h|^{p-2} m^p \\
&\quad
+ \frac1p \int \left\{  \bar a : \frac{\nabla^2(m^p)}{m^p} + 2\, \bar b \cdot \frac{\nabla(m^p)}{m^p} -(p-1)  \bar c     \right\}     |h|^p m^p .
\eal
$$
We can rewrite
$$
\nabla(m^p) = p m^{p-1} \nabla m
$$
and
$$
\nabla^2 (m^p) = (\partial_{ij} m^p)_{1\le i,j\le 3} = p(p-1) m^{p-2} \partial_i m \partial_j m + p m^{p-1} \partial_{ij}m
$$
to obtain
\beqn\label{eq:Lpm}
\frac1p \frac{d}{d t} {\| h \|}_{L^p(m)}^p 
= -(p-1) \int \bar a \nabla h \nabla h |h|^{p-2} m^p
+ \int (\varphi_{m,p} - M\chi_R ) |h|^p m^p,
\eeqn
where $\varphi_{m,p}$ is defined in \eqref{eq:def-phi}.

From Lemma~\ref{lem:phi}, for all $\lambda \geq 0$, we can choose $M$ and $R$ large enough such that $\varphi_{m,p}(v) - M\chi_R (v) \leq -\lambda$. 
Hence, it follows that the operator $(\BB + \lambda )$ is dissipative in $L^p(m)$. Indeed, from \eqref{eq:Lpm} we have
$$
\bal
\frac1p \frac{d}{d t} {\| h \|}_{L^p(m)}^p 
&= -(p-1) \int \bar a\, \nabla h \nabla h\, |h|^{p-2} m^p + \int (\varphi_{m,p} - M\chi_R) |h|^p m^p \\
&\leq -\lambda {\| h \|}_{L^p(m)}^p,
\eal
$$
since the matrix $\bar a$ is positive, and it follows that
\beqn\label{eq:SBexp}
\| \SS_\BB(t) h \|_{L^p(m)} \leq e^{-\lambda t} \| h \|_{L^p(m)}.
\eeqn
\end{proof}

\begin{proof}[Proof of Lemma \ref{lem:hypoBBmu}]
Arguing as in the proof above and denoting $\varphi_{\mu} := \varphi_{\mu^{-1/2},2}$,
that satisfies from \eqref{eq:phi0}
$$
\varphi_\mu(v) = J_{\gamma+2}(v) - \frac12 \ell_1(v) |v|^2 + (\gamma+3) J_\gamma(v),
$$
we obtain
$$
\frac12\frac{d}{dt}\| h \|_{L^2(\mu^{-1/2})}^2 = - \int \bar a \nabla h \nabla h \, \mu^{-1}  + \int(\varphi_{\mu} - M\chi_R)  h^2 \mu^{-1}.
$$
Remark that here we can not conclude as in the proof of Lemma~\ref{lem:hypoBB} because the coefficient of order $\la v \ra^{\gamma+2}$ in $\varphi_\mu$ vanishes in the asymptotic $|v| \to \infty$.

From Lemma~\ref{lem:bar-aij}, there exists $K >0$ such that  $\bar a_{ij} \xi_i \xi_j \geq K \la v \ra^{\gamma} |\xi|^2$. We obtain then
$$
\bal
\frac12\frac{d}{dt}\| h \|_{L^2(\mu^{-1/2})}^2
&\leq - K \int \la v \ra^{\gamma}| \nabla h|^2 \, \mu^{-1}   + \int(\varphi_\mu - M\chi_R)  h^2 \mu^{-1}
\eal 
$$
and, by integration by parts, we also have  
$$
\bal
&\int | \nabla ( \la v \ra^{\gamma/2} \mu^{-1/2} h)|^2 \\
&\qquad
= \int    \la v \ra^{\gamma} \mu^{-1} |\nabla h|^2 + \frac{\gamma^2}{4} |v|^2 \la v \ra^{\gamma-4}\mu^{-1} h^2  + \frac14 |v|^2 \la v\ra^{\gamma} h^2  \mu^{-1} \\ 
&\qquad \quad
+ 2\int \frac{\gamma}{2} h \nabla h v \la v \ra^{\gamma-2} \mu^{-1} 
+ \frac12 h \nabla h v \la v \ra^{\gamma} \mu^{-1}
+ \frac{\gamma}{4} |v|^2 \la v \ra^{\gamma-2} h^2 \mu^{-1}  \\
&\qquad
= \int \la v \ra^{\gamma}  |\nabla h|^2 \mu^{-1}
+ \int \left\{ - \frac14 \la v \ra^{\gamma+2} - \left(\frac54 + \frac{\gamma}{2}\right) \la v \ra^{\gamma} - \frac{\gamma^2}{4} \la v \ra^{\gamma-2} - \frac{\gamma}{4}(4-\gamma) \la v \ra^{\gamma-4}    \right\} h^2 \mu^{-1} .
\eal
$$ 
Finally, it follows that
\beqn\label{eq:reg}
\bal
\frac12\frac{d}{dt}\| h \|_{L^2( \mu^{-1/2} )}^2
&\leq - K \int | \nabla (\la v \ra^{\gamma} \mu^{-1/2} h)|^2   + \int(\tilde \varphi_\mu - M\chi_R)  h^2 \mu^{-1} ,
\eal 
\eeqn
where
$$
\bal
\tilde \varphi_\mu (v) 
&=  \varphi_\mu(v) - \frac14 \la v \ra^{\gamma+2} + C \la v \ra^{\gamma}\\
&= - \frac14 \la v \ra^{\gamma+2} + J_{\gamma+2}(v) - \frac12 \ell_1(v) |v|^2 + (\gamma+3) J_\gamma(v)  + C \la v\ra^{\gamma}.
\eal
$$
Thanks to the asymptotic behaviour of $\ell_1$, $J_{\gamma+2}$ and $J_\gamma$, and arguing as in Lemma~\ref{lem:phi}, we easily get that
$$
\varphi_\mu(v)  \underset{|v|\to\infty}{\sim} -\frac14 \la v \ra^{\gamma+2} \xrightarrow[|v|\to\infty]{} -\infty.
$$
Then, for any $\lambda >0$, we can choose $M,R$ large enough such that $\varphi(v) - M \chi_R(v) \le - \lambda$ for any $v \in \R^3$. We conclude the proof as in the previous lemma.
\end{proof}

\subsection{Regularisation properties}\label{ssec:regularizing}

We are now interested in regularisation properties of the operator $\AA$ and the iterated convolutions of $\AA\SS_\BB$.
Let us recall the operator $\AA$ defined in \eqref{eq:AB},
$$
\AA g = \AA_0 g + M\chi_R g = (a_{ij}\ast g)\partial_{ij}\mu - (c\ast g)\mu + M\chi_R g,
$$
for $M$ and $R$ large enough chosen before.
Thanks to the smooth cut-off function $\chi_R$, for any $q\in[1,+\infty)$, $p\geq q$ and any weight function $m$ satisfying (W), we easily observe that 
\beqn\label{eq:chiR}
\| M \chi_R g \|_{L^q(\mu^{-1/2})} \leq C \| \chi_R \mu^{-1/2} m^{-1} \|_{L^{pq/(p-q)}} \| g\|_{L^{p}(m)} 
\leq C \| g \|_{L^p(m)},
\eeqn
from which we deduce that $M \chi_R \in \BBB(L^p(m), L^q(\mu^{-1/2}))$.

\smallskip

Let us now focus on the operator $\AA_0$. 

\begin{lem}\label{lem:A0}
Let $\gamma\in (-2,0)$ and $q \in [1,2]$. 

\begin{enumerate}[(i)]

\item If $1 \le q < 3/|\gamma|$ then
$$
\| \AA_0 g \|_{L^q(\mu^{-1/2})} \lesssim \| g \|_{L^1(\la v \ra^{\gamma+2})}  +\| g \|_{L^1} +\| g \|_{L^q}.
$$

\item If $\gamma \in (-2, - 3/2 ]$ and $3/|\gamma| \le q \le 2$ then
$$
\| \AA_0 g \|_{L^q(\mu^{-1/2})} \lesssim \| g \|_{L^1(\la v \ra^{\gamma+2})} + \| g \|_{L^{\frac{3}{4 + \gamma}}}.
$$

\end{enumerate}

As a consequence, for any $1 \le p \le 2$ and $m$ satisfying (W1) there hold: 
\begin{itemize}

\item $\AA \in \BBB(L^2(\mu^{-1/2}))$; 

\item $\AA \in \BBB(L^p(m))$ and moreover $\AA \in \BBB(L^p(m), L^p(\mu^{-1/2}))$.

\end{itemize}

\end{lem}

\begin{proof}

For any $1 \le q \le 2$ we write 
$$
\|\AA_0 g \|_{L^q(\mu^{-1/2})} \leq \|(a_{ij}\ast g)\partial_{ij}\mu \|_{L^q(\mu^{-1/2})} 
+ \|(c\ast g)\mu \|_{L^q(\mu^{-1/2})}  ,
$$
and we estimate each term separately.
For the first term, since $|a_{ij}(v-v_*)| \leq C \la v\ra^{\gamma+2} \la v_* \ra^{\gamma+2}$ and $|\partial_{ij} \mu(v)| \le C \la v \ra^2 \mu$, we easily obtain
$$
\bal
\|(a_{ij}\ast g)\partial_{ij}\mu \|_{L^q(\mu^{-1/2})}^q 
&\lesssim   \|  g \|_{L^1(\la v\ra^{\gamma+2})}^q \int \la v \ra^{(\gamma+4)q}  \, \mu^{q/2} \lesssim  \|  g \|_{L^1(\la v\ra^{\gamma+2})}^q .
\eal
$$
For the second term we separate into two cases.

\medskip
\noindent
{$(i)$} Suppose $1 \le q < 3/|\gamma|$. We decompose $c = c_- + c_+$ with $c_- = c {\mathbf 1}_{|\cdot|\le 1}$ and $c_+ = c {\mathbf 1}_{|\cdot| >1}$. We easily bound
$$
| (c_+ * g)(v) | \lesssim \int_{v_*} {\mathbf 1}_{|v-v_*|> 1} \, |v-v_*|^\gamma \, |g_*| \lesssim \| g \|_{L^1},
$$
hence
$$
\| (c_+ * g) \mu \|_{L^q(\mu^{-1/2})} = \| (c_+ * g) \mu^{1/2} \|_{L^q} \lesssim \| g \|_{L^1}.
$$
For the other term, we get 
$$
\bal
\|(c_- * g)\mu \|_{L^q(\mu^{-1/2})}^q
&\lesssim \int_v \left| \int_{v_*} {\mathbf 1}_{|v-v_*| \le 1} \, |v-v_*|^\gamma  g_* \, d v_*\right|^q \,  \mu^{q/2}  \\
&\lesssim\int_v \int_{v_*} {\mathbf 1}_{|v-v_*| \le 1} \,|v-v_*|^{\gamma q}  |g_*|^q \,  \mu^{q/2}  \\
&\lesssim \int_{v_*} \left( \int_v {\mathbf 1}_{|v-v_*| \le 1} \,|v-v_*|^{\gamma q}  \mu^{q/2} \, d v \right)   |g_*|^q  \\
&\lesssim C_\mu \, \| g \|_{L^q}^q,
\eal
$$
where we have used Jensen's inequality at the first line and, in the last line, the integral in $v$ is bounded since $ q < 3/|\gamma|$. This concludes the proof of point {$(i)$}.

\medskip
\noindent
{$(i)$} Now suppose $\gamma \in (-2,-3/2]$ and $3/|\gamma| \le q \le 2$.
We write then
$$
\bal
\|(c * g)\mu \|_{L^q(\mu^{-1/2})} 
&= \|(c * g)\mu^{1/2} \|_{L^q} 
\le \| (c * g) \|_{L^3} \, \| \mu^{1/2} \|_{L^{\frac{3q}{3-q}}} \\
&\lesssim \| g \|_{L^\frac{3}{4+\gamma}} \, \| \mu^{1/2} \|_{L^{\frac{3q}{3-q}}},
\eal
$$
where we have used H\"older's inequality in first line and Hardy-Littlewood-Sobolev inequality in the second one. This gives point $(ii)$.

\medskip

The conclusion of the lemma is a easy consequence of the above estimates and \eqref{eq:chiR}, observing that in the case $(i)$ we have
$$
\| g \|_{L^1(\la v \ra^{\gamma+2})} +\| g \|_{L^1} +\| g \|_{L^q} \lesssim \| g \|_{L^q(m)}
$$ 
and in the case $(ii)$
$$
\| g \|_{L^1(\la v \ra^{\gamma+2})} + \| g \|_{L^{\frac{3}{4 + \gamma}}}
\lesssim \| g \|_{L^q(m)},
$$
for any weight function $m$ satisfying (W).
\end{proof}

We prove now a regularisation estimate for the convolution of $\AA \SS_\BB(t)$.
Let $m_0 := \exp(\kappa_0 \la v \ra^s)$ and $m_1 := \exp(\kappa_1 \la v \ra^s)$ be weight functions satisfying (W) with $\kappa_1 > \kappa_0$, so that $m_0 \le C m_1$. 

\begin{lem}\label{lem:reg}
Let $\gamma \in (-2,0)$.
Consider $1 \le p \le 2$, then there exists $C>0$ such that
\beqn\label{eq:regBB}
\forall\, t\ge 0, \qquad
\| \SS_\BB(t)  \|_{\BBB(L^p(m_1) , L^2(m_0))} \leq C\, t^{-\frac32(\frac1p-\frac12)} \, e^{-\lambda t}.
\eeqn
As a consequence, for all $1\le p \le 2$ and $m$ satisfying assumption (W), 
for any $\lambda' < \lambda$ ($\lambda > 0$ fixed in Lemma~\ref{lem:hypoBB}) we have
\beqn\label{eq:regAB}
\forall\, t\geq 0, \qquad
\| (\AA\SS_\BB)^{*2}(t)  \|_{\BBB(L^p(m) , L^2(\mu^{-1/2}))} \leq C\, e^{-\lambda' t},.
\eeqn

\end{lem}

\begin{proof}[Proof of Lemma \ref{lem:reg}]
We split the proof into two steps.

\medskip
\noindent
\textit{Step 1.} We first prove \eqref{eq:regBB} for $p=1$. 
Consider the equation $\partial_t f = \BB f$. Then from \eqref{eq:Lpm} we have
$$
\frac12\frac{d}{dt}\| f \|_{L^2(m_0)}^2 = - \int \bar a \nabla f \nabla f m^2  + \int(\varphi_{m_0,2} - M\chi_R) m_0^2 f^2
$$
From Lemma~\ref{lem:bar-aij}, there exists $K >0$ such that  $\bar a_{ij} \xi_i \xi_j \geq K \la v \ra^{\gamma} |\xi|^2$, which yields
$$
\bal
\frac12\frac{d}{dt}\| f \|_{L^2(m_0)}^2
&\leq - K \int \la v \ra^{\gamma}| \nabla f|^2 m_0^2   + \int(\varphi_{m_0,2} - M\chi_R) m_0^2 f^2,
\eal 
$$
and, from 
$$
\bal
| \nabla ( \la v \ra^{\gamma} m_0 f)|^2  
&\le C \left\{  \la v \ra^{\gamma} m_0^2 | \nabla  f|^2 
+ C \la v \ra^{\gamma+2s-2}  m_0^2 f^2 \right\},
\eal
$$ 
it follows that
\beqn\label{eq:reg}
\bal
\frac12\frac{d}{dt}\| f \|_{L^2(m_0)}^2
&\leq - K \int | \nabla (\la v \ra^{\gamma} m_0 f)|^2   + \int(\tilde \varphi_{m_0,2} - M\chi_R) m_0^2 f^2 ,
\eal 
\eeqn
where
$$
\tilde \varphi_{m_0,2}(v) =  \varphi_{m_0,2}(v) + C \la v \ra^{\gamma + 2s-2} .
$$
From Lemma~\ref{lem:phi} we easily see that $\tilde\varphi_{m_0,2} \underset{|v|\to +\infty}{\sim} \varphi_{m_0,2}$, then for all $\lambda \geq 0$ we can chose $M$ and $R$ large enough such that $\tilde\varphi_{m_0,2}(v)  - M\chi_R(v) \leq -\lambda$, and moreover estimate \eqref{eq:SBexp} holds.

Applying the following inequality (which can be obtained by H\"older's inequality followed by Sobolev embedding in dimension $d=3$):
$$ 
{\| \la v \ra^\alpha g \|}_{L^2} \leq c_1 {\| \nabla g\|}_{L^2}^{3/5} \, {\| \la v \ra^{5\alpha/2} g\|}_{L^1}^{2/5} 
$$ 
with $g = \la v \ra^{\gamma/2} m_0 f$ and $\alpha = -\gamma/2$ to \eqref{eq:reg}, it follows
\beqn\label{eq:reg2}
\bal
\frac12\frac{d}{dt} {\|  f \|}_{L^2(m_0)}^2
&\leq -K {\| f \|}_{L^2(m_0)}^{10/3} \| \la v \ra^{-3\gamma/4}  f\|_{L^1(m_0)}^{-4/3}
 - \lambda \| f \|_{L^2(m_0)}^2 \\
&\leq -K {\| f \|}_{L^2(m_0)}^{10/3} \| f\|_{L^1(m_1)}^{-4/3}  - \lambda \| f \|_{L^2(m_0)}^2.
\eal 
\eeqn
Recall that the weight functions $m_0$ and $m_1$ satisfy assumption (W), then Lemma~\ref{lem:hypoBB} holds, more precisely, for all $t\ge 0$,
\beqn\label{eq:hypo-m0}
\bal
&\| \SS_\BB(t) f \|_{L^p(m_0)} \leq e^{-\lambda t} \| f\|_{L^p(m_0)} 
\quad\text{ and } \quad 
\| \SS_\BB(t) f \|_{L^p( m_1)} \leq e^{-\lambda t} \| f\|_{L^p(m_1)}  .
\eal
\eeqn
Let us denote now
$$
X(t) := \| f(t) \|_{L^2(m_0)}^2 
\quad\text{ and } \quad 
Y(t) :=  \| f(t) \|_{L^1( m_1)}.
$$
For all $t \ge 0$ we have $Y(y) \leq  Y_0$ from \eqref{eq:hypo-m0}, which together with \eqref{eq:reg2} gives
\beqn\label{eq:dotX}
\dot X(t) \leq - 2K X(t)^{1+2/3} Y_0^{-4/3} - 2 \lambda X(t).
\eeqn
Arguing as \cite[Lemma 3.9]{GMM} we obtain that
$$
\forall\, t\geq 0 \qquad X(t) \leq C \, t^{-3/2} \, e^{-2\lambda t} \, Y_0^2,
$$
which concludes the proof of \eqref{eq:regBB} when $p=1$. Then for any $1<p<2$ we use Riesz-Thorin interpolation theorem, with $\SS_\BB : L^2(m_0) \to L^2(m_0)$ and $\SS_\BB : L^1(m_1) \to L^2(m_0) $, to conclude to \eqref{eq:regBB}.

\medskip
\noindent
\textit{Step 2.} 
Let us prove now \eqref{eq:regAB}. 
From Lemma \ref{lem:A0} we have the following estimates, for any $p\in[1,2]$,
\beqn\label{eq:regAA}
\bal
\| \AA g \|_{L^2(\mu^{-1/2})} \lesssim \| g \|_{L^2(m_0)} ,
\qquad
\| \AA g \|_{L^p(m_0)} \lesssim \| g \|_{L^p(m)} .
\eal
\eeqn
Hence, by \eqref{eq:regAA} and \eqref{eq:regBB}, for $1\leq p \leq 2$, it follows
\beqn
\bal
\| \AA\SS_\BB(t) f \|_{L^2(\mu^{-1/2})} 
&\lesssim \| \SS_\BB(t) f \|_{L^2(m_0)} 
\lesssim t^{-\frac32(\frac1p-\frac12)} e^{- \lambda t}\, \| f \|_{L^p( m_1)}.
\eal
\eeqn
Computing the convolution of $\AA\SS_\BB(t)$ we have 
$$
\bal
\| (\AA\SS_\BB)^{*2}(t) f \|_{L^2(\mu^{-1/2})} 
&\lesssim \int_0^t \|\AA\SS_\BB(t-s) \AA\SS_\BB(s) f \|_{L^2(\mu^{-1/2})} \, ds\\
&\lesssim \int_0^t \|\SS_\BB(t-s) \AA\SS_\BB(s) f \|_{L^2(m_0)} \, ds\\
&\lesssim \int_0^t (t-s)^{-\frac32(\frac1p-\frac12)} e^{-\lambda(t-s)}\,\| \AA\SS_\BB(s) f \|_{L^p(m_1)} \, ds\\
&\lesssim \int_0^t (t-s)^{-\frac32(\frac1p-\frac12)} e^{-\lambda(t-s)}\,\| \SS_\BB(s) f \|_{L^p(m)} \, ds\\
&\lesssim \int_0^t (t-s)^{-\frac32(\frac1p-\frac12)} e^{-\lambda(t-s)}\, e^{-\lambda s} \, \| f \|_{L^p(m)} \, ds\\
&\lesssim \, t^{(\frac74 - \frac{3}{2p} )} \, e^{- \lambda t}\,  \| f \|_{L^p(m)}\\
&\lesssim e^{- \lambda' t} \, \| f \|_{L^p(m)},
\eal
$$
where we have used successively \eqref{eq:regAA}, \eqref{eq:regBB}, \eqref{eq:regAA} and Lemma~\ref{lem:hypoBB} with $1\leq p \le 2$, which concludes the proof.
\end{proof}

\subsection{Proof of Theorem~\ref{thm:trou}}

We are know able to prove Theorem~\ref{thm:trou} that extends to various weighted $L^p$-spaces the semigroup decay estimate known to hold on $L^2(\mu^{-1/2})$ as presented Proposition~\ref{prop:gap}.

\medskip

Let $E = L^2(\mu^{-1/2})$, in which space we already know that there is a spectral gap $\lambda_0 >0$ from Proposition~\ref{prop:gap}, and $\EE = L^p(m)$, for any $p\in[1,2]$ and $m$ satisfying assumption (W). We consider the decomposition $\LL = \AA + \BB$ as in \eqref{eq:AB}. For any $\lambda > 0$, the operator $\BB+\lambda$ is hypo-dissipative in $\EE$ from Lemma~\ref{lem:hypoBB}, moreover $\AA \in \BBB(\EE)$ and $A \in \BBB(E)$ from Lemma~\ref{lem:A0}. Finally, from Lemma \ref{lem:reg} we have that $(\AA\SS_\BB)^{*2}(t) \in \BBB(\EE,E)$ with an exponential decay rate $\| (\AA\SS_\BB)^{*2}(t) \|_{\BBB(\EE,E)} \le C_{\lambda'} \, e^{-\lambda' t}$ for any $\lambda' < \lambda$. 
Then the result of Theorem~\ref{thm:trou} follows from \cite[Theorem 2.13]{GMM}.

\section{A priori estimates}\label{sec:apriori}

The purpose of this section is to establish a priori estimates for the (nonlinear) Landau equation that will be of crucial importance in the proof of the main results in Section~\ref{sec:conv}.

\smallskip

Let us recall the Landau equation that is given by
$$
\partial_t f = Q(f,f)
$$
with
$$
Q(g,f) = \nabla \cdot \{ (a*g)\nabla f - (b*g)f \} = (a_{ij} * g)\partial_{ij} f - (c*g)f.
$$

\subsection{Preliminaries}

Denoting $\bar a_g = a*g$, $\bar b_g  = b * g$, $\bar c_g = c*g$ and considering some weight function $m$, we easily compute
$$
\bal
\int Q(g,f) \, f^{p-1} \, m^p 
&= \int \nabla \cdot \{ \bar a_g \nabla f - \bar b f   \} f^{p-1} \, m^p \\
&= - \int \bar a_g \nabla f \nabla (f^{p-1}) m^p - \int \bar a_g \nabla f \nabla m^p \, f^{p-1} \\
&\quad
+\int \bar b_g f \nabla(f^{p-1}) m^p 
+ \int \bar b_g \nabla m^p \, f^p.
\eal
$$
It follows that
\beqn\label{eq:dtLpm}
\bal
\int Q(g,f) \, f^{p-1} \, m^p 
&= -\frac4p(1-1/p)\int_v \bar a_g \nabla(f^{p/2}) \nabla(f^{p/2})\, m^p \\
&\quad 
+ \int_v \int_{v_*} \Theta_{m,p}(v,v_*)\, g_* \, f^p\, m^p \\
&\quad 
+(1/p-1)\int_v \int_{v_*} c(v-v_*)\, g_* \, f^p\, m^p
\eal
\eeqn
where
\beqn\label{eq:Theta}
\bal
\Theta_{m,p}(v,v_*) &= a(v-v_*) : \frac{D^2 m}{m}(v) + (p-1)\, a(v-v_*) \, \frac{\nabla m}{m}(v) \, \frac{\nabla m}{m}(v) \\
&\quad
+ 2 b(v-v_*) \cdot \frac{\nabla m}{m}(v).
\eal
\eeqn
In the particular case of a polynomial weight $m = \la v \ra^k$, we have
\beqn\label{eq:Theta}
\bal
\Theta_{m,p}(v,v_*) 
&= k |v-v_*|^{\gamma} \la v \ra^{-2} (-2 \la v \ra^2 + 2 \la v_* \ra^2) \\
&\quad
+ k(kp-2) | v-v_*|^{\gamma} \la v \ra^{-4} \left[ |v|^2 |v_*|^2 - (v\cdot v_*)^2 \right].
\eal
\eeqn

\medskip

We recall the following elementary interpolation inequalities.

\begin{lem}\label{lem:ineg}
Let $k,\ell \in \R_+$.
For all $\eps >0$ there is $C_\eps$ such that
$$
{\| g \|}_{\dot H^k }^2 \leq \eps {\| g \|}_{\dot H^{k+\ell} }^2 + C_\eps {\| g \|}_{L^1 }^2,
$$
$$
{\| g \|}_{\dot H^k }^2 \leq \eps {\| g \|}_{\dot H^{k+\ell} }^2 + C_\eps {\| g \|}_{L^2 }^2.
$$
\end{lem}

Moreover, we have an interpolation inequality for weighted Sobolev spaces from \cite{DesWen}:

\begin{lem}\label{lem:interpolation}
For any $\delta,\alpha \ge 0$ and $k \in \R$, there holds
$$
\| f \|_{H^k_l}^2 \le C_\delta \| f \|_{H^{k-\delta}_{l+\alpha}} \, {\| f \|}_{H^{k+\delta}_{l-\alpha}}.
$$
\end{lem}

Now we state a technical lemma that will be useful in the estimates of weighted $L^2$-type norms.

\begin{lem}\label{lem:Kalpha}
Let $0 < \alpha < d$. Consider smooth nonnegative functions $f,g,h : \R^d \to \R$ and define
$$
K_\alpha(f,g,h) := \iint |v-v_*|^{-\alpha} \, f_* \, g \, h \, dv_* \, dv.
$$
Let $\ell_0 \le \alpha$ and  $\ell_1 + \ell_2 = - \ell_0 $, then the following estimates hold:

\begin{enumerate}[(1)]

\item For any $\sigma\ge 0$ such that $2 \sigma <d$ and $2(\alpha - \sigma) <d$ we have
$$
\bal
K_\alpha(f,g,h) 
&\lesssim
\| \la v \ra^{\ell_0} f \|_{L^1} \, 
\| \la v \ra^{\ell_1} g \|_{L^2} \, \| \la v \ra^{\ell_2} h \|_{L^2} 
+ \| \la v \ra^{\ell_0} f \|_{L^1} \, \| \la v \ra^{\ell_1} g \|_{\dot H^{\alpha-\sigma}} \, \| \la v \ra^{\ell_2} h \|_{\dot H^\sigma}.
\eal
$$

\smallskip

\item For any $0 < \sigma < \alpha$ we have
$$
\bal
K_\alpha(f,g,h) 
&\lesssim
\| \la v \ra^{\ell_0} f \|_{L^1} \, \| \la v \ra^{\ell_1} g \|_{L^2} \, \| \la v \ra^{\ell_2} h \|_{L^2} 
+ \| \la v \ra^{\ell_0} f \|_{L^{\frac{d}{d-\alpha+\sigma}}} \, 
\| \la v \ra^{\ell_1} g \|_{H^{\sigma}} \, \| \la v \ra^{\ell_2} h \|_{L^2}.
\eal
$$

\end{enumerate}

\end{lem}

\begin{proof}
Denote $F_* = \la v_* \ra^{\ell_0} |f_*|$, $G = \la v \ra^{\ell_1} |g| $ and $H = \la v \ra^{\ell_2} |h|$ such that $\ell_1+ \ell_2 = -\ell_0$, and split the integral into two parts, $K_1 := \iint \Indiq{|v-v_*|\le 1}$ and $K_2 := \iint \Indiq{|v-v_*| >1}$. 
Then
$$
\bal
K_2
&= \iint \Indiq{|v-v_*| > 1} |v-v_*|^{-\alpha} \la v_* \ra^{-\ell_0} \la v \ra^{\ell_0} \, F_* \, G \, H \, dv_* \, dv  \\
&\lesssim
\iint F_* \, G \, H= \| \la v \ra^{\ell_0} f \|_{L^1} \, \| \la v \ra^{\ell_1} g \|_{L^2} \, \| \la v \ra^{\ell_2} h \|_{L^2},
\eal
$$
where we have used, since $\ell_0 \le \alpha$,
$$
\bal
|v - v_*|^{-\alpha} \Indiq{|v-v_*| > 1} \la v_* \ra^{-\ell_0} \la v \ra^{\ell_0}
\le 2^{\alpha/2} \Indiq{|v-v_*| > 1} \frac{\la v - v_* \ra^{\ell_0}}{|v-v_*|^\alpha}
\le C.
\eal
$$
This gives the first term in the estimates above, both for points (1) and (2).
For the term $K_1$ we split into two cases.

\medskip
\noindent
(1) Using that $\la v_* \ra^{-\ell_0} \la v \ra^{\ell_0} \Indiq{|v-v_*|\le 1} \le C$ we obtain
$$
\bal
K_1 \lesssim \iint |v-v_*|^{-\alpha} F_* \, G \, H = \int_{v_*} F_* \left\{ \int_v |v-v_*|^{-\alpha} \,  G \, H \right\}
\eal
$$
and we need to estimate the integral in $v$. Using Pitt's inequality \cite{Beckner}, for any $\sigma \ge 0$ such that $2\sigma < d$ and $2(\alpha - \sigma) < d$, we get
$$
\bal
\int_v |v-v_*|^{-\alpha} \,  G \, H 
&\le \left(\int |v-v_*|^{-2(\alpha-\sigma)} \,  G^2 \right)^{1/2} 
\left(\int |v-v_*|^{-2\sigma} \,  H^2 \right)^{1/2} \\
&\lesssim \left(\int |\xi|^{2(\alpha-\sigma)} | \hat G |^2 \right)^{1/2}
\left( \int |\xi|^{2\sigma} |\hat H|^2 \right)^{1/2} \\
&\lesssim  \| \la v \ra^{\ell_1} g \|_{\dot H^{\alpha-\sigma}} \, \| \la v \ra^{\ell_2} h \|_{\dot H^\sigma}.
\eal
$$

\medskip
\noindent
(2) Using Hardy-Littlewood-Sobolev inequality, for any $0 < \sigma < \alpha$, we get
$$
\bal
K_1 
\lesssim \iint |v-v_*|^{-\alpha} F_* \, G \, H 
\lesssim \| F \|_{L^{\frac{d}{d- \alpha + \sigma}}} \, \| G H \|_{L^{\frac{d}{d-\sigma}}}.
\eal
$$
Using H\"older's inequality and the Sobolev embedding $H^{\sigma}(\R^d) \hookrightarrow L^{\frac{2d}{d-2\sigma}} (\R^d)$, it follows that
$$
\| G H \|_{L^{\frac{d}{d-\sigma}}} \le \| G \|_{L^{\frac{2d}{d-2\sigma}}} \, \| H \|_{L^2}
\lesssim \| \la v \ra^{\ell_1} g \|_{H^{\sigma}} \, \| \la v \ra^{\ell_2} h \|_{L^2},
$$
which completes the proof. 
\end{proof}

We state next a result from \cite[Proposition 4]{DesVi1} (see also \cite{ALL}) concerning
ellipticity properties of the matrix $a*f$.

\begin{lem}\label{lem:coerc}
Let $\gamma \in [-2,1]$ and $f \in L^1_2 \cap L \log L (\R^d)$. Then there exists $K>0$ depending on $\| f \|_{L^1_2 \cap L \log L}$ such that
$$
(a*f)(v) \geq K \la v \ra^{\gamma} \,  I_d.
$$

\end{lem}

The proof of this result is stated in \cite{DesVi1} in the case $\gamma \in (0,1]$, however we easily observe that the result is also valid for $\gamma \ge -2$ by following the proof.

\subsection{Moments estimates}

The moments of solutions to the Landau equation in the case of soft potentials is known to be propagated linearly in time, as is stated in \cite[Section 2.4, p.~73]{Villani-BoltzmannBook}. We give however a proof of this fact for the sake of completeness and because we shall need a precise estimate in order to use it later for the stretched exponential moments in Lemma~\ref{lem:momentsexp}.

\begin{lem}\label{lem:moments}
Let $\gamma \in (-2,0)$, $f_0 \in L^1_2 \cap L \log L$ and consider a weak solution $ f \in L^\infty([0,\infty); L^1_2 \cap L \log L)$ to the Landau equation associated to $f_0$. Suppose further that $f_0 \in L^1_l$ for some $l>2$. Then, at least formally, there exists a constant $C>0$ depending on $\| f \|_{L^\infty([0,\infty); L^1_2)}$ and $\| f_0 \|_{L^1_l}$ (but not on $l$) such that 
$$
\forall\, t \ge 0 \qquad
\| f(t) \|_{L^1_l} \leq C \, \alpha(l)\, (1+t)
$$
with
$$
\alpha(l) :=
\left\{
\bal
&l^2, &\quad l \le 4,\\
&\frac{ l^{2 - \frac{4}{\gamma+2}} } {l-4} \, \left( \frac{l-4}{l+\gamma-2} \right)^{\frac{l-4}{\gamma+2}} \, l^{\frac{l}{\gamma+2}}, &\quad l >4.
\eal
\right.
$$
\end{lem}

\begin{proof}
The equation for the moments is 
$$
\frac{d}{d t} \| f \|_{L^1_{l}} =  \iint  |v-v_*|^{\gamma} \left\{ - 2l + 2l \la v \ra^{-2}\la v_* \ra^2 + l(l-2) \la v \ra^{-4}[|v|^2 |v_*|^2 - (v\cdot v_*)^2]  \right\} f_* \, f \, \la v \ra^l.
$$
Because of the singularity of $|v-v_*|^{\gamma}$, we split it into two parts $|v-v_*|^{\gamma} \Indiq{|v-v_*|\geq 1}\,$ and $|v-v_*|^{\gamma} \Indiq{|v-v_*|\leq 1}\,$, denoting respectively $T_1$ and $T_2$ each associated term.
Using that $|v|^2 |v_*|^2 - (v\cdot v_*)^2 \leq \la v \ra^2 \la v_* \ra^2$, we obtain for $T_1$ that
$$
\bal
T_1 &\leq -2l \iint  |v-v_*|^{\gamma} \Indiq{|v-v_*|\geq 1} \, f_* \, f  \, \la v \ra^{l} \\
&\quad
+ l^2 \iint  |v-v_*|^{\gamma} \Indiq{|v-v_*|\geq 1}\, \la v \ra^{l-2}\la v_* \ra^2 f_* \, f ,
\eal
$$
from which we get
\beqn\label{eq:T1}
T_1 \leq - K l  \| f \|_{L^1_{l+\gamma}} + C  l^2 \| f \|_{L^1_{l-2}},
\eeqn
for constants $K,C>0$, using the conservation of mass and energy.

For the term $T_2$, we write
$$
\bal
T_2 &=  
l \iint |v-v_*|^{\gamma}\Indiq{|v-v_*|\leq 1}\, \la v \ra^{l-2} \left\{ - 2 \la v \ra^2 + 2 \la v_* \ra^2   \right\} f_* \, f \\
& + l(l-2) \iint |v-v_*|^{\gamma}\Indiq{|v-v_*|\leq 1}\, \la v \ra^{l-4} 
\left\{ |v|^2 |v_*|^2 - (v\cdot v_*)^2 \right\}  f_* \, f=: T_{21} + T_{22}.
\eal
$$
Using H\"older's inequality
$$
\bal
&\iint f f_* |v-v_*|^{\gamma}\Indiq{|v-v_*|\leq 1}\, \la v \ra^{l-2} \la v_* \ra^2 \\
&\qquad 
\leq  \left( \iint f f_* |v-v_*|^{\gamma}\Indiq{|v-v_*|\leq 1}\, \la v \ra^{l} \right)^{(l-2)/l}
\left( \iint f f_* |v-v_*|^{\gamma} \Indiq{|v-v_*|\leq 1}\, \la v_* \ra^{l}\right)^{2/l}\\
&\qquad
=\iint f f_* |v-v_*|^{\gamma}\Indiq{|v-v_*|\leq 1}\, \la v \ra^{l}
\eal
$$
and this implies $T_{21} \leq 0$. Moreover, using the inequality $|v|^2 |v_*|^2 - (v\cdot v_*)^2 \leq |v| |v_*| |v-v_*|^2$, we obtain
$$
\bal
T_{22} &\leq C l^2 \iint f f_* |v-v_*|^{\gamma+2} \Indiq{|v-v_*|\leq 1}\, \la v \ra^{l-3} \la v_* \ra \\
& \leq C l^2 \| f \|_{L^1_{1}} \, \| f \|_{L^1_{l-3}} \leq C  l^2 \| f \|_{L^1_{l-2}},
\eal
$$ 
where we have used $|v-v_*|^{\gamma+2} \Indiq{|v-v_*|\leq 1}\, \leq 1$ and $\| f \|_{L^1_{1}}$ uniformly bounded.
Gathering $T_1$ and $T_2$, it follows that
$$
\frac{d}{dt} \| f \|_{L^1_{l}} \leq - K  l \| f \|_{L^1_{l+\gamma}} + C l^2 \| f \|_{L^1_{l-2}}.
$$

If $l\le 4$ then $\| f \|_{L^1_{l-2}}$ is uniformly bounded and we easily conclude.

\smallskip

Consider then $l >4$.
Since $\gamma>-2$, denoting $r = (l+\gamma-2)/(l-4) >1$ and $r'=r/(r-1) = (l+\gamma -2)/(\gamma+2)$, it follows by H\"older and Young's inequality that
$$
\bal
\| f \|_{L^1_{l-2}}  
\leq \| f \|_{L^1_2}^{1/r'} \, \| f \|_{L^1_{l+\gamma}}^{1/r} 
\leq  \frac{1}{r'} \, \eta^{-\frac{l-4}{\gamma+2}}  \| f \|_{L^1_2} + \frac{\eta}{r} \| f \|_{L^1_{l+\gamma}},
\eal
$$
for all $\eta >0$. We obtain
$$
\frac{d}{d t} \| f \|_{L^1_{l}}
\le - K  l \| f \|_{L^1_{l+\gamma}} + C l^2 \frac{\eta}{r} \| f \|_{L^1_{l+\gamma}}
+ C \frac{l^2}{r'} \, \eta^{-\frac{l-4}{\gamma+2}}  \| f \|_{L^1_2} 
\le C l^{2 - \frac{4}{\gamma+2}}  \, l^{\frac{l}{\gamma+2}} \, \frac{r^{-\frac{l-4}{\gamma+2}}}{r'} \, \| f \|_{L^1_2},
$$
choosing $\eta = Kr/(Cl)$, from which we conclude to
$$
\| f(t) \|_{L^1_{l}} \leq C\, \alpha(l) \, (1+t).
$$
\end{proof}

As a consequence of the above result, we deduce a similar linearly growing estimate for some stretched exponential moments.

\begin{lem}\label{lem:momentsexp}
Let $\gamma \in (-2,0)$, $f_0 \in L^1_2 \cap L \log L$ and consider a weak solution $ f \in L^\infty([0,\infty); L^1_2 \cap L \log L)$ to the Landau equation associated to $f_0$. 
Suppose further that $f_0 \in L^1(e^{\kappa \la v \ra^s})$ with $\kappa >0$ and $0<s < 2+\gamma$.
Then, at least formally, there exists a constant $C>0$ depending on $\| f \|_{L^\infty([0,\infty); L^1_2)}$, $\| f_0 \|_{L^1(e^{\kappa \la v \ra^s})}$, $\kappa$ and 
$s$ such that 
$$
\forall\, t \ge 0, \qquad
\| f(t) \|_{L^1(e^{\kappa \la v \ra^s})} \leq C (1+t).
$$

\end{lem}

\begin{proof}
Write
$$
e^{\kappa \la v \ra^s} = \sum_{j=0}^\infty \kappa^j \frac{\la v \ra^{js}}{j!}
$$
and then, using Lemma~\ref{lem:moments}, we have
$$
\bal
\| f(t) \|_{L^1(e^{\kappa \la v \ra^s})} 
&= \sum_{j=0}^\infty \frac{\kappa^j}{j!} \int f(t) \la v \ra^{js} \\
&\leq \sum_{j=0}^\infty \frac{\kappa^j}{j!} \left\{ C \, \alpha(sj) \, t +        \int f_0 \la v \ra^{js}  \right\} = C t \sum_{j=0}^\infty \frac{\kappa^j}{j!} \, \alpha(sj) + \| f_0 \|_{L^1(e^{\kappa \la v \ra^s})},
\eal
$$
and we only need to prove that the sum is finite. Let $j_0 \in \N$ such that $s j_0 \le 4 < s(j_0+1)$. Then we have
$$
\sum_{j=j_0+1}^\infty \frac{\kappa^j}{j!} \, \alpha(sj)
= \sum_{j=j_0+1}^\infty \kappa^j \, \frac{  (sj)^{2 - \frac{4}{\gamma+2}}} {sj-4} \, \left( \frac{sj-4}{sj+\gamma-2} \right)^{\frac{sj-4}{\gamma+2}} \, s^{\frac{sj}{\gamma+2}} \, \frac{ j^{( \frac{s}{\gamma+2}) j}}{ j! },
$$
which is finite if $s < \gamma+2$.
\end{proof}

\subsection{Regularity estimates}

We shall establish coercivity estimates for the Landau operator $Q$, which are inspired by some similar estimates obtained by Wu~\cite{Wu} and Alexandre, Lao and Lin~\cite{ALL}.

\begin{lem}\label{lem:Qpoly}
Let $\gamma \in (-2,0)$. Then for smooth functions $f$ and $g$, there are constants $K, C >0$ depending on $\| f \|_{L^1_2 \cap L\log L}$ such that: 

\begin{enumerate}[(i)]

\item If $0 \le k \le (\gamma+3)/2$ then
$$
\left\la Q(f,g) , g   \la v \ra^{2k} \right\ra 
\le - K \|  g   \|_{\dot H^1_{k + \gamma/2}}^2
+ C  \|  g \|_{L^2_{k + \gamma/2}}^2.
$$

\item If $k > (\gamma+3)/2$ then
$$
\bal
\left\la Q(f,g) , g  \la v \ra^{2k} \right\ra 
&\le - K\|  g   \|_{\dot H^1_{k + \gamma/2}}^2 
- K' \|  g \|_{L^2_{k+\gamma/2}}^2 
+ C  \|  g \|_{L^2_{k-1}}^2.
\eal
$$

\end{enumerate}

\end{lem}

\begin{proof}
From \eqref{eq:dtLpm} and \eqref{eq:Theta} for $m = \la v \ra^k$, we obtain
\beqn
\bal
\left\la Q(f,g) , g  \la v \ra^{2k} \right\ra 
&= - \int_v ( a*f) \nabla g \nabla g \, \la v \ra^{2k} \\
&\quad
+( \gamma+3 -2k) \int_v \int_{v_*} |v-v_*|^\gamma  \, f_* \, g^2 \, \la v \ra^{2k} \\
&\quad 
+2k \int_v \int_{v_*} |v-v_*|^\gamma \la v \ra^{-2} \la v_* \ra^{2}  \, f_* \, g^2 \, \la v \ra^{2k} \\
&\quad
+2k(k-1) \int_v \int_{v_*} |v-v_*|^\gamma \la v \ra^{-4}[|v|^2 |v_*|^2 - (v \cdot v_*)^2]  \, f_* \, g^2 \, \la v \ra^{2k} \\
& =: I_1 + I_2 + I_3+I_4.
\eal
\eeqn

\medskip

For the first term $I_1$, we use the coercivity property of $\bar a$,
since $f\in L^1_2 \cap L\log L$, we have from Lemma~\ref{lem:coerc} that
$$
\bar a (v) = (a * f)(v) \ge K \la v \ra^{\gamma} I_3 .
$$
Then we get
$$
\bal
I_1 
&\le - K \| \la v \ra^{\gamma/2+k} \, \nabla g \|_{L^2}^2 = - K \| g \|_{\dot H^1_{k + \gamma/2}}^2,
\eal
$$
which can also be written as
$$
I_1 
\le - K \| g \|_{\dot H^1_{k + \gamma/2}}^2 
\le  - K \| \la v \ra^{\gamma/2+k}  g \|_{\dot H^1}^2 + C \| g \|_{L^2_{k+\gamma/2-1}}^2 .
$$

For the second term $I_2$, we split into two cases.
If $k \le (\gamma+3)/2$ we have, from Lemma \ref{lem:Kalpha} and the interpolation inequality from Lemma~\ref{lem:ineg}, that
$$
\bal
|I_2|
&\lesssim  \iint |v-v_*|^{\gamma} \, f_* \, g^2 \, \la v \ra^{2k} \\
&\lesssim  \| \la v \ra^{-\gamma} f \|_{L^1}  \big\{C_\epsilon \| \la v \ra^{\gamma/2 + k} g\|_{L^2}^2 + \epsilon \| \la v \ra^{\gamma/2+k} g \|_{\dot H^1}^2 \big\},
\eal
$$
for any $\epsilon>0$.
However, if $ k > (\gamma+3)/2$, we get
$$
I_2 \le - K' \int_v \int_{v_*} |v-v_*|^\gamma f_* \, g^2 \, \la v \ra^{2k}
\le - K \| \la v \ra^{\gamma/2+k} g \|_{L^2}^2.
$$

Finally, using that $|v|^2 |v_*|^2 - (v \cdot v_*)^2 \le \la v\ra^2 \la v_*\ra^2$ we easily get
$$
I_3 + I_4 \lesssim  \int_v \int_{v_*} |v-v_*|^\gamma \la v \ra^{-2} \la v_* \ra^{2}\, f_* \, g^2 \, \la v \ra^{2k}.
$$
Then, arguing as in the proof of Lemma \ref{lem:Kalpha} (term $K_1$ in that lemma) and using again Lemma~\ref{lem:ineg}, it follows that
$$
\bal
I_3 + I_4
&\lesssim  \| \la v \ra^2 f \|_{L^1} \, \| \la v \ra^{k-1} g \|_{\dot H^{-\gamma/2}}^2 \\
&\lesssim  C_\epsilon \| \la v \ra^{k-1} g \|_{L^2}^2 + \epsilon \| \la v \ra^{k-1} g \|_{\dot H^1}^2 \\
&\lesssim  C_\epsilon \| \la v \ra^{k-1} g \|_{L^2}^2 + \epsilon \| \la v \ra^{\gamma/2 +k} g \|_{\dot H^1}^2   .
\eal
$$
for any $\epsilon >0$.
We then conclude gathering all previous estimates and taking $\epsilon >0$ small enough.
\end{proof}

We also prove an upper bound for $Q$ in the following lemma. It is worth mentioning that 
He~\cite{He} obtain similar estimates by a different method.

\begin{lem}\label{lem:Q2poly}
Let $\gamma \in (-2,0)$ and consider smooth functions $f$, $g$ and $h$. Then for any $\ell_1 + \ell_2 = \gamma + 2$ we have
$$
\left| \left\la Q(f,g) , h \la v \ra^{2k} \right\ra \right| 
\lesssim  \| f \|_{L^1_{\gamma+2}} \, \| g \|_{H^{1}_{\ell_1+k}}    \, \| h \|_{H^1_{\ell_2+k}}.
$$
\end{lem}

\begin{proof} We write
$$
\bal
\left\la Q(f,g) , g  \la v \ra^{2k} \right\ra =  \int \nabla \{(a * f)\nabla g \} h \la v \ra^{2k} - \int \nabla \{(b*f) g \} h \la v \ra^{2k} =: T_1 + T_2. 
\eal
$$
For the first term, we easily obtain, since $|a(v-v_*)| \lesssim |v-v_*|^{\gamma+2} \lesssim \la v_*\ra^{\gamma+2} \la v \ra^{\gamma+2}$, that
$$
\bal
T_1 
&\lesssim 
\iint |v-v_*|^{\gamma+2} \, |f_*| \, |\nabla g| \, |\nabla h| \, \la v \ra^{2k}
+ \iint |v-v_*|^{\gamma+2} \, |f_*| \, |\nabla g| \, |h| \, \la v \ra^{2k-1} \\
&\lesssim
\| f \|_{L^1_{\gamma+2}} \Big\{ \| \nabla g \|_{L^2_{\ell_1 +k}}  \| \nabla h \|_{L^2_{\ell_2 + k}} 
+ \| \nabla g \|_{L^2_{\ell_1 +k-1/2}}  \| h \|_{L^2_{\ell_2 + k-1/2}}
   \Big\}.
\eal
$$
Moreover, for the second term, it follows that
$$
\bal
T_2
&\lesssim 
\iint |v-v_*|^{\gamma+1} \, |f_*| \, |g| \, |\nabla h| \, \la v \ra^{2k}
+ \iint |v-v_*|^{\gamma+1} \, |f_*| \, |g| \, |h| \, \la v \ra^{2k-1}.
\eal
$$
Now we investigate two different cases.
If $\gamma+1 \ge 0$, using $|v-v_*|^{\gamma+1} \lesssim \la v_*\ra^{\gamma+1} \la v \ra^{\gamma+1}$ we obtain
$$
T_2 \lesssim
\| f \|_{L^1_{\gamma+1}} 
\Big\{
\| g \|_{L^2_{\ell_1 +k - 1/2}} \|  \nabla h \|_{L^2_{\ell_2+k-1/2}}
+\| g \|_{L^2_{\ell_1 +k - 1}} \|   h \|_{L^2_{\ell_2+k-1}}
\Big\}.
$$
On the other hand, if $\gamma+1 < 0$, i.e.\ $-2 < \gamma < -1$, we use Lemma~\ref{lem:Kalpha} to get
$$
\bal
T_2 
&\lesssim \| f \|_{L^1_{-(\gamma+1)}} 
\Big\{  \| \la v \ra^{\ell_1+k-1/2}  g \|_{L^2} \, \| \la v \ra^{\ell_2+k-1/2} \nabla h \|_{L^2} 
+ \| \la v \ra^{\ell_1+k-1/2}  g \|_{\dot H^{-(\gamma+1)}} \, \| \la v \ra^{\ell_2+k-1/2} \nabla h \|_{L^2} \\
&\qquad\qquad\qquad\qquad
+
\| \la v \ra^{\ell_1+k-1}  g \|_{L^2} \, \| \la v \ra^{\ell_2 +k -1} h \|_{L^2} 
+ \| \la v \ra^{\ell_1+k-1}  g \|_{\dot H^{-(\gamma+1)}} \, \| \la v \ra^{\ell_2+k-1}  h \|_{L^2}
\Big\}.
\eal
$$
We conclude gathering the above estimates.
\end{proof}

We prove now some estimates for weighted $L^2$ and Sobolev norms.

\begin{prop}\label{prop:L2}
Let $\gamma \in (-2,0)$, $f_0 \in L^1_2 \cap L \log L (\R^3)$ and consider a weak solution $f \in L^\infty([0,\infty) ; L^1_2 \cap L \log L)$ of the Landau equation associated to $f_0$. Then, at least formally, there holds:

\begin{enumerate}[(1)]

\item Let $0 \le k \le 2 + 3\gamma/4$. Then for any $t_0 > 0$ there is $C=C(t_0) >0$ such that
$$
\sup_{t \ge t_0} \left\{ \|  f(t) \|_{L^2_k}^2 + \int_{t}^{t+1} \|  f(\tau) \|_{H^1_{k+\gamma/2}}^2 \, d\tau   \right\} \le C.
$$

\item Let $k > 2 + 3\gamma/4$ and suppose $f_0 \in L^1_{k-3\gamma/4}$. Then for any $t_0 > 0$ there is $C=C(t_0)>0$ such that
$$
\forall\, t \ge t_0, \quad
 \|  f(t) \|_{L^2_k}^2 + \int_{0}^{t} \| f(\tau) \|_{H^1_{k+\gamma/2}}^2 \, d\tau    \le C(1+t).
$$
\end{enumerate}

\end{prop}

\begin{proof}

{\it(1)} From Proposition \ref{lem:Qpoly}, for $0\le k \le 2 + 3\gamma/4$, we have
\beqn\label{eq:dtL2k}
\bal
\frac{d}{dt} \| f \|_{L^2_k}^2 
&\le - K \| f \|_{\dot H^1_{k + \gamma/2}}^2 + C \|  f \|_{L^2_{k + \gamma/2}}^2 \\
&\le - K \|  \la v \ra^{\gamma/2 + k} f \|_{\dot H^1}^2 + C \|  f \|_{L^2_{k + \gamma/2}}^2.
\eal
\eeqn
Using the following inequality (obtained by H\"older and Sobolev's inequalities in dimension $d=3$),
\beqn\label{eq:Nashpoids}
\| \la v \ra^\alpha u \|_{L^2} \le C \| \nabla u \|_{L^2}^{3/5} \, \| \la v \ra^{5\alpha/2} u \|_{L^1}^{2/5},
\eeqn
we obtain that, choosing $\alpha = - \gamma/2$,
\beqn\label{eq:dtL2kbis}
\frac{d}{dt} \| f \|_{L^2_k}^2 
\le - K \|  f \|_{L^1_{k-3\gamma/4}}^{-4/3} \, \|  f \|_{L^2_{k}}^{2+4/3} + C \|  f \|_{L^2_{k + \gamma/2}}^2.
\eeqn
Since $\|  f(t) \|_{L^1_{k-3\gamma/4}} \le \|  f(t) \|_{L^1_2}$ is uniformly bounded in time, we finally get, applying Young's inequality for the last term,
$$
\frac{d}{dt} \| f \|_{L^2_k}^2 
\le - K \| f \|_{L^2_k}^{2+4/3} + C,
$$
from which we deduce by standard arguments that for any $t_0 >0$ there exists $C=C(t_0)>0$ such that
$$
\sup_{t \ge t_0}  \|  f(t) \|_{L^2_k}^2 \le C.
$$
Coming back to \eqref{eq:dtL2k} we also obtain
$$
\sup_{t \ge t_0} \int_{t}^{t+1} \|  f(\tau) \|_{H^1_{k+\gamma/2}}^2 \, d\tau \le C.
$$

\medskip
\noindent
{\it (2)} 
Remark that $k > 2 + 3\gamma/4 > (\gamma+3)/2$, hence Proposition~\ref{lem:Qpoly} yields
\beqn\label{eq:dtL2kk}
\bal
\frac{d}{dt} \| f \|_{L^2_k}^2 
&\le - K \|  f \|_{\dot H^1_{k + \gamma/2}}^2
-K \|  f \|_{L^2_{k+\gamma/2}}^2 
+ C \| f \|_{L^2_{k-1}}^2 \\
&\le - K \|  \la v \ra^{k + \gamma/2}  f \|_{\dot H^1}^2
-K \|  f \|_{L^2_{k+\gamma/2}}^2 
+ C \| f \|_{L^2_{k-1}}^2.
\eal
\eeqn
Using the interpolation inequality, for any $\delta,\eps >0$,
$$
\|  f \|_{L^2_{\alpha}}^2 \le \eps \| f \|_{L^2_{\alpha+\delta}}^2 + C_\eps \| f \|_{L^2}^2,
$$
for $\eps >0$ small enough and \eqref{eq:Nashpoids}, we finally get
\beqn\label{eq:dtL2kkbis}
\bal
\frac{d}{dt} \| f \|_{L^2_k}^2 
&\le - K \| \la v \ra^{\gamma/2 + k} f \|_{\dot H^1}^2
-K \|  f \|_{L^2_{k+\gamma/2}}^2 
+ C \| f \|_{L^2}^2 \\
&\le - K \| f \|_{L^1_{k-3\gamma/4}}^{-4/3} \, \| f \|_{L^2_k}^{2+4/3}
-K \|  f \|_{L^2_{k+\gamma/2}}^2 
+ C \| f \|_{L^2}^2.
\eal
\eeqn
Now we fix some $t_0 >0$. From point (1) we know that there exists $C>0$ such that $\sup_{t \ge t_0/2} \| f(t) \|_{L^2}^2 \le C$. Moreover, since $f_0 \in L^1_{k - 3\gamma/4}$, Lemma~\ref{lem:moments} implies $\| f(t) \|_{L^1_{k-3\gamma/4}} \le C(1+t)$ for any $t\ge 0$, so that 
$\sup_{[t_0/2 , 3t_0/2]} \| f(t) \|_{L^1_{k-3\gamma/4}} < \infty$. Writing \eqref{eq:dtL2kkbis} for $ t \in[t_0/2, 3t_0/2]$, we obtain by standard arguments that for any $t_1 > t_0/2$ we have 
$\sup_{[t_1 , 3t_0/2]} \| f(t) \|_{L^2_{k}} < \infty$. Coming back to \eqref{eq:dtL2kkbis} and neglecting the negative terms, we obtain
$$
\forall\, t \ge t_0, \qquad
\frac{d}{dt} \| f \|_{L^2_k}^2 
\le  C,
$$
from which we have
$$
\|  f(t) \|_{L^2_k}^2 \le C\int_{t_0}^t d\tau + \| f(t_0) \|_{L^2_k}^2
\le C(1+t).
$$
We also deduce 
$$
\int_{t_0}^t  \|  f(\tau) \|_{H^1_{k+\gamma/2}}^2 \, d\tau \le C(1+t)
$$
coming back to \eqref{eq:dtL2kk} and using the previous bound.
\end{proof}

\begin{prop}\label{prop:H12n}
Let $\gamma \in (-2,0)$, $f_0 \in L^1_2 \cap L \log L (\R^3)$ and consider a weak solution $f \in L^\infty([0,\infty) ; L^1_2 \cap L \log L)$ of the Landau equation associated to $f_0$. 
Then, at least formally, there hold:

\begin{enumerate}[(1)]

\item Suppose $f_0 \in L^1_{k - 5\gamma/4}$. Then for any $t_0 > 0$ there exists $C=C(t_0)>0$ such that, for all $t \ge t_0$,
$$
\| f(t) \|_{H^1_k}^2 + \int_{t_0}^{t} \| f(\tau) \|_{H^2_{k+ \gamma/2}}^2 \, d\tau 
\le C (1 + t)^2.
$$

\item Suppose $f_0 \in L^1_{k -7\gamma/4} \cap L^1_{2l -k - 7\gamma/4}$ with $l := \max(\gamma+4, k + \gamma/2 + 2)$. Then for any $t_0 > 0$ there exists $C=C(t_0)>0$ such that, for all $t \ge t_0$,
$$
\| f(t) \|_{H^2_k}  \le C (1 + t)^{7/2}.
$$

\item Suppose further that the weak solution satisfies $f \in L^\infty([0, \infty) ; L^1_l)$ for any $l\ge 0$. Then for all $ t_1 > 0$, any $k\ge 0$ and $n \in \N$, there is $C=C(t_1)>0$ such that
$$
\sup_{t \ge  t_1} \| f(t) \|_{H^n_k} \le C.
$$

\end{enumerate}

\end{prop}

\begin{proof}

{(1)} 
Let $\alpha \in \N^3$ be a multi-index such that $|\alpha|=1$ and denote $g = \partial^\alpha f$, which satisfies the equation
$$
\partial_t g = Q(f, g) + Q(g , f),
$$
and then we easily compute
\beqn\label{eq:dgdt}
\frac12 \frac{d}{dt} \|   g \|_{L^2_k}^2 = \left\la  Q(f, g), g \la v \ra^{2k} \right\ra +
\left\la Q(g,f), g  \la v \ra^{2k} \right\ra =: T_1 + T_2.
\eeqn
From Lemma~\ref{lem:Qpoly} we observe that
\beqn\label{T1}
\bal
T_1 
&\le 
- K \|  g \|_{\dot H^1_{k+\gamma/2}}^2 + C \|  g\|_{L^2_{k+\gamma/2}}^2\\
&\le 
- K \| \la v \ra^{\gamma/2+k} g \|_{\dot H^1}^2 + C \|  g\|_{L^2_{k+\gamma/2}}^2.
\eal
\eeqn
For the second term, we write $T_2 = T_{22} + T_{21}$ with
$$
T_{21} := \int \nabla \cdot \{ (a*g) \nabla f \} g \la v \ra^{2k}
\quad\text{and}\quad
T_{22} := -\int \nabla \cdot \{(b*g) f \} g \la v \ra^{2k}.
$$
Integrating by parts and using the symmetry of $a$, it follows that
$$
\bal
T_{21} 
&= - \int  (a_{ij}* g) \partial_j f \partial^\alpha(\partial_i f) \la v \ra^{2k}
- \int  (a_{ij}* g) \partial_j f \partial^\alpha f  \partial_i \la v \ra^{2k} \\
&= \frac12 \int (\partial^{\alpha} \partial^{\alpha} a_{ij} * f) \partial_i f \partial_j f \la v \ra^{2k} 
+\frac12 \int (\partial^{\alpha} a_{ij} * f) \partial_i f \partial_j f  \partial^{\alpha}\la v \ra^{2k}
- \int  (a_{ij}* g) \partial_j f \partial^\alpha f  \partial_i \la v \ra^{2k} \\
&\lesssim \iint |v-v_*|^\gamma f_* \, |\nabla f|^2 \la v \ra^{2k} 
+ \iint |v-v_*|^{\gamma+1} f_* \, |\nabla f|^2 \la v \ra^{2k-1}. 
\eal
$$
Using Lemma~\ref{lem:Kalpha}-(1), it follows that
$$
\bal
\iint |v-v_*|^\gamma f_* \, |\nabla f|^2 \la v \ra^{2k}
&\lesssim \| \la v \ra^{-\gamma} f \|_{L^1} \Big\{
\| \la v \ra^{\gamma/2+k} \nabla f \|_{L^2}^2
+\| \la v \ra^{\gamma/2+k} \nabla f \|_{\dot H^{-\gamma/2}}^2 \Big\},
\eal
$$
and also
$$
\bal
&\iint |v-v_*|^{\gamma+1} f_* \, |\nabla f|^2 \la v \ra^{2k-1}\\
&\qquad\quad\lesssim 
\left\{
\bal
& \| \la v \ra^{\gamma+1} f \|_{L^1} \, \| \la v \ra^{\gamma/2 + k} \nabla f \|_{L^2}^2, \quad &\text{if } \gamma +1 \ge 0;\\
&\| \la v \ra^{-\gamma-1} f \|_{L^1} \left(
\| \la v \ra^{\gamma/2+k} \nabla f \|_{L^2}^2
+\| \la v \ra^{\gamma/2+k} \nabla f \|_{\dot H^{-(\gamma+1)/2}}^2
\right), \quad &\text{if } \gamma + 1 <0.
\eal
\right.
\eal
$$
Using the uniform in time bound of $\| f(t) \|_{L^1_{2}}$, the previous estimates yield
\beqn\label{T21}
\bal
T_{21} 
&\lesssim \| \la v \ra^{\gamma/2+k} \nabla f \|_{L^2}^2
+\| \la v \ra^{\gamma/2+k} \nabla f \|_{\dot H^{-\gamma/2}}^2\\
&\lesssim C(\epsilon)\|  f \|_{\dot H^1_{k + \gamma/2}}^2 
+ \epsilon \|  f \|_{\dot H^2_{k + \gamma/2}}^2,
\eal
\eeqn
for any $\epsilon >0$, thanks to the interpolation Lemma~\ref{lem:ineg}.
For the term $T_{22}$ we obtain
$$
T_{22}
\lesssim \iint |v-v_*|^\gamma f_* \, f \, |\nabla^2 f| \, \la v \ra^{2k}
+ \iint |v-v_*|^\gamma f_* \, f \, |\nabla f| \, \la v \ra^{2k-1} =: T_{221} + T_{222}.
$$
Thanks to Lemma~\ref{lem:Kalpha}-(1) again, we get
\beqn\label{T222}
\bal
T_{222} 
&\lesssim \| \la v \ra^{-\gamma} f \|_{L^1} 
\Big\{ \| \la v \ra^{\gamma/2 + k} f \|_{L^2} \| \la v \ra^{\gamma/2 + k} \nabla f \|_{L^2}
+ \| \la v \ra^{\gamma/2 + k} f \|_{\dot H^{-\gamma-1}} \| \la v \ra^{\gamma/2 + k} \nabla f \|_{\dot H^{1}}  \Big\}\\
& \lesssim  
 \| \la v \ra^{\gamma/2 + k} f \|_{L^2}^2 + \| \la v \ra^{\gamma/2 + k} \nabla f \|_{L^2}^2
+ C(\epsilon)\| \la v \ra^{\gamma/2 + k} f \|_{\dot H^{-\gamma-1}}^2 
+ \epsilon \| \la v \ra^{\gamma/2 + k} \nabla f \|_{\dot H^{1}}^2 \\
& \lesssim  
 C(\epsilon)\|  f \|_{H^1_{k + \gamma/2}}^2 
+ \epsilon \|  f \|_{\dot H^2_{k + \gamma/2}}^2,
\eal
\eeqn
for any $\epsilon >0$, where we have used Young's inequality and the interpolation Lemma~\ref{lem:ineg}.

For the last term $T_{221}$, we split into two different cases.

\medskip
\noindent
\textit{Case (i): $\gamma \in (-3/2,0)$.} 
Using again Lemma~\ref{lem:Kalpha}-(1) (remark that here we need $\gamma >-3/2 $), it follows
\beqn\label{T221-cas1}
\bal
T_{221} 
&\lesssim \| \la v \ra^{-\gamma} f \|_{L^1} 
\Big\{ \| \la v \ra^{\gamma/2 + k} f \|_{L^2} \| \la v \ra^{\gamma/2 + k} \nabla^2 f \|_{L^2}
+ \| \la v \ra^{\gamma/2 + k} f \|_{\dot H^{-\gamma}} \| \la v \ra^{\gamma/2 + k} \nabla^2 f \|_{L^2}  \Big\} \\
&\lesssim  
 C(\epsilon) \| \la v \ra^{\gamma/2 + k} f \|_{L^2}^2 
+ C(\epsilon)\| \la v \ra^{\gamma/2 + k} f \|_{\dot H^{-\gamma}}^2 
+ \epsilon \| \la v \ra^{\gamma/2 + k} \nabla^2 f \|_{L^2}^2 \\
&\lesssim  
 C(\epsilon)\|  f \|_{H^1_{k + \gamma/2}}^2 
+ \epsilon \|  f \|_{\dot H^2_{k + \gamma/2}}^2  .
\eal
\eeqn
Now, coming back to \eqref{eq:dgdt}, gathering the above estimates \eqref{T1}-\eqref{T21}-\eqref{T222}-\eqref{T221-cas1} and taking $\epsilon>0$ small enough, we obtain
\beqn\label{eq:H1k}
\bal
\frac{d}{dt} \|  f \|_{\dot H^1_k}^2 
&\le - K \| f \|_{H^2_{k+\gamma/2}}^2 + C \| f \|_{H^1_{k+\gamma/2}}^2\\
&\le - K \| f \|_{H^2_{k+\gamma/2}}^2 
+C \|  f \|_{L^2_{k+\gamma/2}}^2,
\eal
\eeqn
where we have used Lemma~\ref{lem:ineg} again.

We fix some $t_0 >0$. Since $f_0 \in L^{1}_{k - 5\gamma/4}$, we can use Proposition~\ref{prop:L2} to get that there is $C>0$ such that
$$
\int_{t_0/2}^{t_0} \| f(\tau) \|_{H^1_{k}}^2 \, d\tau < C
\quad\text{and}\quad
\| f(t) \|_{L^2_{k - \gamma/2}}^2 \le C(1+t), \quad \forall\, t \ge t_0 / 2,
$$
from which we can choose some $t_1 \in [t_0/2 , t_0]$ such that $ \| f(t_1) \|_{H^1_{k}}^2 < \infty$. Now we integrate \eqref{eq:H1k} from $t_1$ to $t$ to obtain
$$
\bal
\| f(t) \|_{H^1_k}^2 + K \int_{t_1}^t \| f(\tau) \|_{H^2_{k + \gamma/2}}^2 \, d\tau
&\le C \int_{t_1}^t \| f(\tau) \|_{L^2_{k + \gamma/2}}^2 \, d\tau + \| f(t_1) \|_{H^1_k}^2\\
&\le C (1+t)^2,
\eal
$$
which concludes the case $\gamma \in (-3/2,0)$.

\medskip
\noindent
\textit{Case (ii): $\gamma \in (-2,-3/2]$.} 
In this case, Lemma~\ref{lem:Kalpha}-(2) implies, for any $0 < \sigma < |\gamma|$, any $\ell_0 \le -\gamma$ and $\ell_1 + \ell_2 = - \ell_0$, that
$$
\bal
T_{221} 
&\lesssim \| \la v \ra^{-\gamma} f \|_{L^1} \,
\| \la v \ra^{\gamma/2 + k} f \|_{L^2} \| \la v \ra^{\gamma/2 + k} \nabla^2 f \|_{L^2} \\
&\quad
+  \| \la v \ra^{\ell_0} f \|_{L^{\frac{3}{3+\gamma + \sigma}}} \, \| \la v \ra^{k + \ell_1} f \|_{H^{\sigma}} \| \la v \ra^{k + \ell_2} \nabla^2 f \|_{L^2}
=: A + B.
\eal
$$
The first term $A$ can be easily bounded by
\beqn\label{A}
\bal
A 
&\lesssim C(\epsilon) \| f \|_{L^1_{-\gamma}}^2 \, \| f \|_{L^2_{k+\gamma/2}}^2 + \epsilon \| f \|_{\dot H^2_{k+\gamma/2}}^2\\
&\lesssim C(\epsilon) \| f \|_{L^2_{k+\gamma/2}}^2 + \epsilon \| f \|_{\dot H^2_{k+\gamma/2}}^2,
\eal
\eeqn
for any $\epsilon>0$, and it remains to estimate the last term $B$. We choose $\sigma$ verifying $-3/2 - \gamma < \sigma$ so that $3/(3+\gamma+\sigma) < 2$. Moreover, we choose $\ell_2 = \gamma/2$ and $\ell_0 = 2 + 3\gamma/4$, which implies $\ell_1 = - 2 - 5\gamma/4$. We interpolate $L^{\frac{3}{3+\gamma + \sigma}}$ between $L^1$ and $L^2$, which yields
$$
\| \la v \ra^{2+3\gamma/4} f \|_{L^{\frac{3}{3+\gamma + \sigma}}}
\le  \| \la v \ra^{2+3\gamma/4} f \|_{L^1}^{1 + \frac23(\gamma+\sigma)}  \, 
\| \la v \ra^{2+3\gamma/4} f \|_{L^2}^{-\frac23(\gamma+\sigma)}.
$$
Since we have $-3/2 - \gamma < \sigma < - \gamma$ and $\gamma \in (-2,-3/2]$, we can choose $\sigma = 1/2$. Using the fact that $\| \la v \ra^{k - 2 - 5\gamma/4} f \|_{H^{1/2}} \lesssim \| f \|_{H^{1/2}_{k - 2 - 5\gamma/4}}$ and applying Lemma~\ref{lem:interpolation} twice, it follows
$$
\| f \|_{H^{1/2}_{k - 2 - 5\gamma/4}} 
\lesssim \| f \|_{L^2_{k - 8/3 - 11 \gamma / 6}}^{3/4} \,  \| f \|_{H^2_{k + \gamma/2}}^{1/4}.
$$
This implies, using the uniform in time bound of $\| f(t) \|_{L^1_{2+3\gamma/4}}$ and Young's inequality, the following estimate
\beqn\label{B}
\bal
B
&\lesssim \|  f \|_{L^1_{2+3\gamma/4}}^{1 + \frac23(\gamma+\sigma)}  \, 
\|  f \|_{L^2_{2+3\gamma/4}}^{-\frac23(\gamma+\sigma)} \,
\| f \|_{L^2_{k - 8/3 - 11 \gamma / 6}}^{3/4} \,  \| f \|_{H^2_{k + \gamma/2}}^{5/4}\\
&\lesssim C(\epsilon)\|  f \|_{L^1_{2+3\gamma/4}}^{1 + \frac89(2\gamma+1)}  \, 
\|  f \|_{L^2_{2+3\gamma/4}}^{-\frac89(2\gamma+1)} \,
\| f \|_{L^2_{k - 8/3 - 11 \gamma / 6}}^{2} 
+ \epsilon  \| f \|_{H^2_{k + \gamma/2}}^{2} \\
&\lesssim C(\epsilon)  
\|  f \|_{L^2_{2+3\gamma/4}}^{-\frac89(2\gamma+1)} \,
\| f \|_{L^2_{k - 8/3 - 11 \gamma / 6}}^{2} 
+ \epsilon  \| f \|_{H^2_{k + \gamma/2}}^{2}.
\eal
\eeqn
We can now come back to \eqref{eq:dgdt}. Gathering the above estimates \eqref{T1}-\eqref{T21}-\eqref{T222}-\eqref{A}-\eqref{B} and taking $\epsilon>0$ small enough, we obtain
\beqn\label{eq:H1k-bis}
\bal
\frac{d}{dt} \|  f \|_{\dot H^1_k}^2 
&\le - K \| f \|_{H^2_{k+\gamma/2}}^2 + C \| f \|_{H^1_{k+\gamma/2}}^2 + C \|  f \|_{L^2_{2+3\gamma/4}}^{-\frac89(2\gamma+1)} \,
\| f \|_{L^2_{k - 8/3 - 11 \gamma / 6}}^{2} \\
&\le - K \| f \|_{H^2_{k+\gamma/2}}^2 
+C \|  f \|_{L^2_{k+\gamma/2}}^2 + 
C \|  f \|_{L^2_{2+3\gamma/4}}^{-\frac89(2\gamma+1)} \,
\| f \|_{L^2_{k - 8/3 - 11 \gamma / 6}}^{2} ,
\eal
\eeqn
where we have used Lemma~\ref{lem:ineg}.

We fix some $t_0 >0$ and argue in a similar way as in the previous case. 
First of all, thanks to Proposition~\ref{prop:L2} there holds $\sup_{t \ge  t_0/2} \| f(t) \|_{L^2_{2+3\gamma/4}} \le C$, hence we can rewrite \eqref{eq:H1k-bis} starting from $t_0/2$ as 
\beqn\label{eq:H1k-bis2}
\bal
\frac{d}{dt} \|  f \|_{\dot H^1_k}^2 
&\le - K \| f \|_{H^2_{k+\gamma/2}}^2 
+C \|  f \|_{L^2_{k+\gamma/2}}^2 + 
C \| f \|_{L^2_{k - 8/3 - 11 \gamma / 6}}^{2} , \qquad \forall\, t \ge t_0 / 2,\\
&\le - K \| f \|_{H^2_{k+\gamma/2}}^2 
+C \|  f \|_{L^2_{k+\gamma/2}}^2 + 
C \| f \|_{L^2_{k - \gamma/2}}^{2} ,
\eal
\eeqn
using the fact that $- 8/3 - 11 \gamma / 6 \le - \gamma/2$ because $\gamma >-2$.
Since $f_0 \in L^{1}_{k - 5\gamma/4}$, we can use Proposition~\ref{prop:L2} to deduce that there is $C>0$ such that
$$
\int_{t_0/2}^{t_0} \| f(\tau) \|_{H^1_{k}}^2 \, d\tau \le  C
\quad\text{and}\quad
\| f(t) \|_{L^2_{k - \gamma/2}}^2 \le C(1+t), \quad \forall\, t \ge t_0 / 2,
$$
from which we can choose some $t_1 \in [t_0/2 , t_0]$ such that $ \| f(t_1) \|_{H^1_{k}}^2 < \infty$. Now we integrate \eqref{eq:H1k-bis2} from $t_1$ to $t$, then we obtain
$$
\bal
\| f(t) \|_{H^1_k}^2 + K \int_{t_1}^t \| f(\tau) \|_{H^2_{k + \gamma/2}}^2 \, d\tau
&\le C \int_{t_1}^t \| f(\tau) \|_{L^2_{k - \gamma/2}}^2 \, d\tau + \| f(t_1) \|_{H^1_k}^2\\
&\le C (1+t)^2,
\eal
$$
and the case $\gamma \in (-2,-3/2]$ is complete.

\medskip
\noindent
{(2)}
Let $\beta \in \N^3$ be a multi-index with $|\beta|=2$ and denote $g = \partial^\beta f$. Then $g$ satisfies
$$
\bal
\frac12\frac{d}{dt} \|  g \|_{L^2_k}^2 
&= \la Q(f,g), g \la v \ra^{2k}\ra 
+\sum_{\underset{1\le|\beta_1|\le2, 0\le |\beta_2| \le 1}{\beta_1 + \beta_2 = \beta}} C^{\beta_1}_{\beta_2}
\la Q(\partial^{\beta_1} f, \partial^{\beta_2} f), g \la v \ra^{2k} \ra \\
&=: I_1 + \sum I_2^{\beta_1,\beta_2} .
\eal
$$
For the first term, we have from Lemma~\ref{lem:Qpoly} that
$$
\bal
I_1 
&\le - K \|  g \|_{\dot H^1_{k + \gamma/2}}^2 + C \|  g\|_{L^2_{k+\gamma/2}}^2 \\
&\le - K \| \la v \ra^{\gamma/2+k} g \|_{\dot H^1}^2 + C \|  g\|_{L^2_{k+\gamma/2}}^2.
\eal
$$
For the second one, we use Lemma~\ref{lem:Q2poly} to obtain
$$
\bal
I_2^{\beta_1,\beta_2} 
&\lesssim \| \partial^{\beta_1} f \|_{L^1_{\gamma+2}} \, 
\| \partial^{\beta_2} f \|_{H^1_{k + \gamma/2 + 2}} \,
\| \partial^{\beta} f \|_{H^1_{k + \gamma/2}} \\
&\lesssim
\|  f \|_{H^{|\beta_1|}_{\gamma+4}} \, 
\|  f \|_{H^{|\beta_2| + 1}_{k + \gamma/2 + 2}} \,
\|  f \|_{H^3_{k + \gamma/2}},
\eal
$$
using H\"older's inequality. Then, denoting $l = \max(\gamma+4, k + \gamma/2 + 2)$ and using Lemma~\ref{lem:interpolation}, it follows that
$$
\bal
I_2^{\beta_1,\beta_2}
&\lesssim \|  f \|_{H^{1}_{l}} \, \|  f \|_{H^{2}_{l}} \, \| f \|_{H^3_{k + \gamma/2}}\\
&\lesssim \|  f \|_{H^{1}_{l}} \, \|  f \|_{H^{1}_{2l-k-\gamma/2}}^{1/2} \, \| f \|_{H^3_{k + \gamma/2}}^{3/2} \\
&\lesssim C_{\epsilon} \|  f \|_{H^{1}_{2l-k-\gamma/2}}^6  + \epsilon \| f \|_{H^3_{k + \gamma/2}}^{2},
\eal
$$
for any $\epsilon>0$.
Gathering the above estimates and taking $\epsilon>0$ small enough, we finally obtain the following differential inequality
\beqn\label{eq:H2k}
\bal
\frac{d}{dt} \|  f \|_{\dot H^2_k}^2 
&\le - K \| f \|_{H^3_{k+\gamma/2}}^2 + C \| f \|_{L^2_{k+\gamma/2}}^2 + C  \|  f \|_{H^{1}_{2l-k-\gamma/2}}^6 .
\eal
\eeqn
We fix some $t_0 > 0$. Since $f_0 \in L^1_{k - 7\gamma/4}$ we obtain from point (1) that
$$
\int_{t_0/2}^{t_0} \| f(\tau) \|_{H^2_{k}}^2 \, d\tau < \infty,
$$
from which we deduce that there is some $t_1 \in [t_0/2, t_0]$ such that $\| f(t_1) \|_{H^2_{k}} < \infty$. Moreover, since $f_0 \in L^1_{2l - k  - 7\gamma/4}$, we also get from point (1) that
$$
\forall\, t > t_0/2, \quad \| f(t) \|_{H^1_{2l-k-\gamma/2}}^2 \le C(1+t)^2.
$$
Coming back to \eqref{eq:H2k} and integrating from $t_1$, it yields, for any $t \ge t_1$,
$$
\bal
\| f(t) \|_{H^2_k}^2 + K \int_{t_1}^t \| f(\tau) \|_{H^3_{k + \gamma/2}}^2 \, d\tau
&\le C \int_{t_1}^t \| f(\tau) \|_{L^2_{k + \gamma/2}}^2 \, d\tau + C \int_{t_1}^t \| f(\tau) \|_{H^1_{2l-k-\gamma/2}}^6 \, d\tau + \| f(t_1) \|_{H^2_k}^2\\
&\le C \int_{t_1}^t (1+\tau) \, d\tau + C \int_{t_1}^t (1+\tau)^6 \, d\tau + C  \\
&\le C (1+t)^7,
\eal
$$
which concludes the proof.

\medskip
\noindent
{(3)}
Suppose now that $f \in L^\infty([0,\infty); L^1_l)$ for any $l \ge 0$. We recall that we obtain in Proposition~\ref{prop:L2} (see equations~\eqref{eq:dtL2k} and \eqref{eq:dtL2kbis}) the following differential inequality
$$
\bal
\frac{d}{dt} \| f \|_{L^2_k}^2 
&\le - K \| f \|_{H^1_{k + \gamma/2}}^2 + C \| f \|_{L^2_{k + \gamma/2}}^2\\
&\le - K  \| f \|_{L^1_{k-3\gamma/4}}^{-4/3} \,  \| f \|_{L^2_k}^{2 + 4/3} 
+ C \| f \|_{L^2_{k + \gamma/2}}^2.
\eal
$$
Now since $\| f(t) \|_{L^1_{k-3\gamma/4}}$ is bounded uniformly in time, arguing as in Proposition~\ref{prop:L2}, we obtain from last inequality that for any $t_1 >0$, for any $k \ge 0$, there exists $C=C(t_1)>0$ such that
\beqn\label{eq:L2unif}
\sup_{t \ge t_1} \left\{  \| f(t) \|_{L^2_k}^2  + \int_{t}^{t+1} \| f(\tau) \|_{H^1_k}^2 \, d\tau \right\}\le C. 
\eeqn

Let us now investigate the $\dot H^1_k$-norm. Coming back to \eqref{eq:H1k} if $\gamma\in (-3/2,0)$ or to \eqref{eq:H1k-bis2} if $\gamma \in (-2,-3/2]$, and using Lemma~\ref{lem:interpolation}, we get
$$
\bal
\frac{d}{dt} \| f \|_{\dot H^1_k}^2 
&\le - K  \| f \|_{H^2_{k+\gamma/2}}^{2} + C \| f \|_{L^2_{k - \gamma/2}}^2\\
&\le - K \| f \|_{L^2_{k-\gamma/2}}^{-2} \, \| f \|_{H^1_k}^{4} + C \| f \|_{L^2_{k - \gamma/2}}^2.
\eal
$$
We fix some $t_1>0$. Thanks to \eqref{eq:L2unif} we have $\sup_{t \ge t_1/2} \| f(t) \|_{L^2_l} \le C$ for any $l\ge 0$, then, arguing as before, there is $C=C(t_1)$ such that
$$
\sup_{t \ge t_1} \| f(t) \|_{H^1_k}^2 + \int_{t}^{t+1} \| f(\tau) \|_{H^2_{k}}^2 \, d\tau \le C.
$$

We conclude the proof by induction. Assume that for some integer $n \ge 2$, for any $t_1>0$ and any $k\ge 0$ we have
$$
\sup_{t \ge t_0} \| f(t) \|_{H^{n-1}_k}^2 \le C.
$$
Arguing as in point (2) we obtain that
$$
\frac{d}{dt} \| f \|_{\dot H^{n}_k}^2 
\lesssim - K \| f \|_{H^{n+1}_{k + \gamma/2}}^2 + C\| f \|_{H^n_{k + \gamma/2}}^2
+ C \sum_{\underset{1\le|\beta_1|\le n, 0\le |\beta_2| \le n-1}{|\beta_1| + |\beta_2| = n}}
I^{\beta_1,\beta_2}.
$$
where
$$
I^{\beta_1,\beta_2} \lesssim
\|  f \|_{H^{|\beta_1|}_{\gamma+4}} \, 
\|  f \|_{H^{|\beta_2| + 1}_{k + \gamma/2 + 2}} \,
\|  f \|_{H^{n+1}_{k + \gamma/2}}.
$$
If $(|\beta_1|,|\beta_2|)=(1,n-1)$ or $(|\beta_1|, |\beta_2|)=(n,0)$ then, using Lemma~\ref{lem:interpolation}, it follows
$$
\bal
I^{\beta_1,\beta_2} 
&\lesssim 
\|  f \|_{H^{1}_{l}} \, 
\|  f \|_{H^{n}_{l}} \,
\|  f \|_{H^{n+1}_{k + \gamma/2}}\\
&\lesssim
\|  f \|_{H^{1}_{l}} \, 
\|  f \|_{H^{n-1}_{2l}}^{1/2} \,
\|  f \|_{H^{n+1}_{k + \gamma/2}}^{3/2} 
\lesssim
C_\epsilon + \epsilon \|  f \|_{H^{n+1}_{k + \gamma/2}}^2,
\eal
$$
for any $\epsilon>0$, using the induction hypothesis. In all the other cases, $2 \le |\beta_1| \le n-1$ and $1 \le |\beta_2| \le n-2$, we get
$$
\bal
I^{\beta_1,\beta_2} 
&\lesssim 
\|  f \|_{H^{n-1}_{l}}^2 \, 
\|  f \|_{H^{n+1}_{k + \gamma/2}}
\lesssim
C_\epsilon + \epsilon \|  f \|_{H^{n+1}_{k + \gamma/2}}^2.
\eal
$$
Taking $\epsilon>0$ small enough and iterating Lemma~\ref{lem:interpolation}, we obtain the differential inequality, for some $l, \eta >0$, 
$$
\bal
\frac{d}{dt} \| f \|_{H^{n}_k}^2 
&\lesssim - K \| f \|_{H^{n+1}_{k + \gamma/2}}^2 + C\| f \|_{H^n_{k + \gamma/2}}^2 + C \\
&\lesssim - K \| f \|_{L^2_{l}}^{-\eta} \, \| f \|_{H^n_k}^{2 + \eta} + C\| f \|_{H^n_{k + \gamma/2}}^2 +C,\\
&\lesssim - K \| f \|_{H^n_k}^{2 + \eta} + C,
\eal
$$
using the induction hypothesis, from which it follows that for any $t_1>0$ and any $k\ge 0$ there exists $C=C(t_1)>0$ such that
$$
\sup_{t \ge t_1} \| f(t) \|_{H^n_k} \le C.
$$
\end{proof}

\section{Convergence to equilibrium}\label{sec:conv}

\subsection{Polynomial in time convergence}\label{ssec:poly}

Toscani and Villani \cite{TosVi-slow} have proved a polynomial rate in the trend to equilibrium for the Landau equation for \emph{mollified soft potentials}, i.e. 
replacing the function $a(z) = |z|^{\gamma+2} \Pi(z)$ by a mollified version truncating the singularity at the origin, given by
$$
\widetilde a(z) = \widetilde\Psi(z) \Pi(z)
\qquad\text{with}\qquad
c_\Psi \la z \ra^\gamma \le \frac{\widetilde\Psi(z)}{|z|^2} \le C_\Psi \la z \ra^{\gamma},
$$
for some constants $c_\Psi, C_\psi >0$.
Their strategy was based on two ingredients: a functional inequality relating the entropy and the entropy dissipation functional stated in Lemma~\ref{lem:D<H} (which is also valid in our case of \emph{true} soft potentials), and a priori estimates for the evolution of moments and weighted Sobolev norms for the Landau equation associated with $\widetilde a(z)$ (which of course do not hold in our case and we shall use the new a priori estimates proven in Section~\ref{sec:apriori}).

The entropy - entropy dissipation inequality is given in the following result.

\begin{lem}[{\cite[Proposition 4]{TosVi-slow}}]\label{lem:D<H}
Let $a^{\dag}(z) = \Psi^{\dag}(z) \Pi(z)$ where $\Psi^{\dag}$ verifies
$$
c_{\Psi^{\dag}} \la z \ra^\gamma \le \frac{\Psi^{\dag}(z)}{|z|^2},
$$
for some constant $c_{\Psi^{\dag}} >0$, 
and consider the associated entropy-dissipation functional
$$
D_{a^{\dag}} = \frac12 \iint \Psi^{\dag}(|v-v_*|) \left| \Pi(v-v_*) \left( \frac{\nabla f}{f}-\frac{\nabla f_*}{f_*}   \right)   \right|^2 f_* f \, dv_* \, dv.
$$
Then, for all $k >0$ and all $f$ satisfying \eqref{f0}, there is $C_k (f) >0$ depending on $k$ and $H(f)$ such that
$$
D_{a^{\dag}} (f) \ge C_k(f) \, H(f | \mu )^{1 - \gamma/k} \, \{ \| f \|_{L^1_{k+2}} + J_{k+2}(f)  \}^{\gamma/k}
$$
where
$$
J_{k+2}(f) = \| \la v \ra^{k+2} \nabla \sqrt f \|_{L^2}^2.
$$
\end{lem}

As a consequence of this functional inequality and the fact that
$$
\frac{d}{dt} H(f | \mu) = - D(f),
$$
where $H(f | \mu) = \int f \log(f / \mu)$ is the relative entropy of $f$ with respect to $\mu$, we get the following result.

\begin{cor}[{\cite[Corollary 4.1]{TosVi-slow}}]\label{cor:rate}
If for some $k>0$ we have
$$
\| f(t)\|_{L^1_{k+2}} + J_{k+2}(f(t)) \le C (1+t)^\theta, \qquad \theta < \frac{k}{|\gamma|},
$$
then the following estimate holds true
$$
H(f(t) | \mu ) \le C (1+t)^{-\frac{k}{|\gamma|} + \theta}.
$$
\end{cor}

In order to estimate the evolution of the quantity $J_{k+2}(f(t))$, it is proven in \cite{TosVi-slow} that this quantity can be reduced to weighted Sobolev norms. More precisely, they first prove that
$$
J_{k+2}(f) \lesssim I({\la v \ra^{k+2} f} )  + \| f \|_{L^1_k}
$$
where $I(g)$ is the Fisher information define by
$$
I(g) := \int \frac{|\nabla g|^2}{g} = 4 \int |\nabla \sqrt g|^2.
$$
Finally they prove the following inequality in \cite[Lemma 1]{TosVi-slow}: for any $\eps >0$ there is $C_\eps >0$ such that
$$
I(g) \le C_\epsilon \, \| g \|_{H^2_{3/2 + \epsilon}},
$$
so that at the end we get
$$
\| f \|_{L^1_{k+2}} + J_{k+2}(f) \lesssim \| f \|_{H^2_{k+7/2 + \epsilon}}.
$$

\medskip

Now we are in position to prove the polynomial in time convergence in Theorem~\ref{thm:main0}.

\begin{proof}[Proof of Theorem \ref{thm:main0}]
This theorem is a consequence of Proposition~\ref{prop:H12n} and Corollary~\ref{cor:rate}.
Indeed, remark that Lemma~\ref{lem:D<H} also holds in our case of true soft potentials with $a(z)=|z|^{\gamma+2}\Pi(z)$ given by \eqref{eq:aij}. Then since $f_0 \in L^1_{k + 8 -3\gamma/4} \cap L \log L$ with $k > 7 |\gamma| /2$, the a priori estimate in Proposition~\ref{prop:H12n}-(2) (here one should use approximate solutions of the Landau equation as in \cite{Vi2} in order to give a completely rigorous proof) implies that for any $t_0>0$ it holds
$$
\| f(t) \|_{H^2_{k+4}} \le C (1+t)^{7/2}.
$$
We conclude the proof applying Corollary~\ref{cor:rate}.
\end{proof}

\bigskip

As a consequence of Theorem~\ref{thm:main0} we can improve the slowly increasing a priori bounds for $L^1$ moments in Lemmas~\ref{lem:moments} and \ref{lem:momentsexp}, 
obtaining uniform in time estimates, as done in \cite{DesMo} for the Boltzmann equation.

\begin{prop}\label{prop:momentuniform}
Let $\gamma \in (-2,0)$ and $f_0 \in L^1_2 \cap L \log L$. Consider a global weak solution $f \in L^\infty ([0,\infty); L^1_2 \cap L \log L)$ to the Landau equation. 
\begin{enumerate}[(1)]

\item Suppose that 
$f_0 \in L^1_{2\ell} \cap  L^1_{k+8-3\gamma/4}$ with $\ell>2$ and $k > 11 |\gamma| /2$. Then
$$
\sup_{t\ge 0} \| f(t) \|_{L^1_\ell} \le C.
$$

\item Suppose that 
$f_0 \in L^1(e^{2 \kappa \la v \ra^s}) $ with $\kappa >0$, $0<s<2$ with $s < \gamma + 2$. Then we have
$$
\sup_{t\ge 0} \| f(t) \|_{L^1(e^{\kappa \la v \ra^s})} \le C.
$$
\end{enumerate}

\end{prop}

\begin{proof}
(1) We write, using Lemma~\ref{lem:moments} and Theorem~\ref{thm:main0}
$$
\bal
\| f(t) \|_{L^1_\ell} 
&\le \| f(t) - \mu \|_{L^1_\ell} + \| \mu \|_{L^1_\ell} \\
&\le  \| f(t) - \mu \|_{L^1}^{1/2} \, \| f(t) - \mu \|_{L^1_{2\ell}}^{1/2} + C \\
&\le C (1+t)^{-\frac{k}{4|\gamma|} + \frac78} (1+t)^{\frac12} + C \\
&\le C.
\eal
$$

\medskip
(2) Using Lemma~\ref{lem:momentsexp} and Theorem~\ref{thm:main0}, for some $k> 0$ large enough we have
$$
\bal
\| f(t) \|_{L^1(e^{\kappa \la v \ra^s})} 
&\le \| f(t) - \mu \|_{L^1(e^{\kappa \la v \ra^s})} + \| \mu \|_{L^1(e^{2\kappa \la v \ra^s})} \\
&\le  \| f(t) - \mu \|_{L^1}^{1/2} \, \| f(t) - \mu \|_{L^1(e^{2\kappa \la v \ra^s})}^{1/2} + C \\
&\le C (1+t)^{-\frac{k}{4|\gamma|} + \frac78} (1+t)^{1/2} + C \\
&\le C.
\eal
$$
\end{proof}

\subsection{Exponential in time convergence}\label{ssec:exp}

We are able now to conclude the proof of Theorem~\ref{thm:main}. Recall that in this setting we suppose $\gamma \in (-1,0)$ and $f_0 \in  L \log L \cap L^1(e^{\kappa \la v \ra^s})$ with $\kappa >0$ and $ - \gamma < s < 2 + \gamma$. Let us denote $m = e^{\bar \kappa \la v \ra^s}$ with $\bar \kappa \le \kappa/10$, which satisfies assumption (W).

We write $h(t) = f(t) - \mu$ that satisfies
$$
\left\{
\bal
& \partial_t h = \LL h + Q(h,h) \\
& h_{|t=0} = h_0.
\eal
\right.
$$
Since $\Pi_0 h_0 = 0$ and $\Pi_0 Q(h_0,h_0)= 0 $, for all $t\ge 0$, we also have $\Pi_0 h(t) = 0$ 
and $\Pi_0 Q(h(t), h(t))=0$, thanks to the conservation laws. By Duhamel's principle it follows
\beqn\label{eq:duhamel}
h(t) = S_{\LL}(t) h_0 + \int_0^t S_{\LL}(t-\tau) Q(h(\tau), h(\tau)) \, d\tau.
\eeqn

Before starting the proof of the main theorem, let us state a result that will be useful in the sequel.

\begin{lem}\label{lem:Q-Lpm}
Let $\gamma\in (-2,0)$, $p\in [1,+\infty)$ and $m$ be a weight function. 
Then, if $1 \le q < 3/|\gamma|$, it holds
$$
\|  Q(g,f) \|_{L^p(m)}
\lesssim  
\| g\|_{L^{1}(\la v\ra^{\gamma+2})} \, \| \nabla^2 f\|_{ L^{p}(m\la v\ra^{\gamma+2})} 
+ \| g \|_{L^1} \, \| f \|_{L^p(m)}
+ \| g\|_{L^{q/(q-1)}} \, \| f \|_{L^{p}(m)}.
$$
\end{lem}

\begin{proof}
Since $\gamma\in (-2,0)$, using $|v-v_*|^{\gamma+2} \lesssim \la v_* \ra^{\gamma+2} \la v \ra^{\gamma+2}$ we easily obtain
\beqn\label{eq:mouT1}
\bal
\| (a_{ij}*g) \partial_{ij} f \|_{L^p(m)}^p 
& \lesssim \int_v \left| \int_{v_*} |v-v_*|^{\gamma+2} g_* \, dv_* \right|^p \, |\partial_{ij} f(v)|^p \, m^p \, dv \\
&\lesssim \|  g\|_{L^1(\la v\ra^{\gamma+2})}^p  \| \nabla^2 f\|_{ L^{p}(m\la v\ra^{\gamma+2})}^p.
\eal
\eeqn
Now let us denote $c_- = c \Indiq{|\cdot| \le 1}$ and $c_+ = c \Indiq{|\cdot| > 1}$.
We can also obtain
$$
\bal
\|(c_+ * g)  f \|_{L^p(m)}^p 
&\lesssim \int \left| \int |v-v_*|^\gamma \Indiq{|v-v_*|>1} \, |g_*| \, dv_* \right|^p |f|^p \, m^p \, dv \\
&\lesssim \| g \|_{L^1}^p \, \| f \|_{L^p(m)}^p,
\eal
$$
and
$$
\bal
\|(c_- * g)  f \|_{L^p(m)}^p 
&\lesssim \int \left| \int |v-v_*|^\gamma \Indiq{|v-v_*| \le 1} \, |g_*| \, dv_* \right|^p |f|^p \, m^p \, dv \\
&\lesssim \int \left\{ \left(\int |v-v_*|^{\gamma q} \Indiq{|v-v_*| \le 1} \, dv_* \right)^{1/q} \left( \int|g_*|^{q'} \, dv_* \right)^{1/q'} \right\}^p |f|^p \, m^p \, dv \\
&\lesssim \| g \|_{L^{q'}}^p \, \| f \|_{L^p(m)}^p
\eal
$$
using H\"older's inequality and if $1 \le q < 3/|\gamma|$.
\end{proof}

\begin{proof}[Proof of Theorem \ref{thm:main}]
We split the proof into several steps.

\medskip
\noindent
\textit{Step 1.}
Since $f_0 \in L^1_2 \cap L \log L \cap L^1(e^{\kappa \la v \ra^s})$ we can apply Theorem~\ref{thm:main0} that implies
\beqn\label{eq:poly-conv}
\forall\, t\ge 0, \quad
\| h(t) \|_{L^1} = \| f(t) - \mu \|_{L^1} \le C (1+t)^{-\theta}, \quad \forall\, \theta>0.
\eeqn
Moreover we get, using Lemma~\ref{lem:momentsexp},
$$
\bal
\| h(t) \|_{L^1(e^{\frac{\kappa}{2} \la v \ra^s})} 
&\le \| h(t) \|_{L^1} \, \| h(t) \|_{L^1(e^{\kappa \la v \ra^s})} \\
&\le C(1+t)^{-\theta} \left(  \| f(t) \|_{L^1(e^{\kappa \la v \ra^s})} + \| \mu \|_{L^1(e^{\kappa \la v \ra^s})} \right) \\
&\le C(1+t)^{-\theta} \left( (1+t) + C_\mu \right)
\le C (1+t)^{-\theta+1}.
\eal
$$

\medskip
\noindent
\textit{Step 2.}
Since $f_0 \in L^1_\ell \cap L \log L$ for any $\ell \ge 0$, Proposition~\ref{prop:momentuniform} implies
$$
f \in L^\infty([0,\infty) ; L^1_\ell) \quad \forall\, \ell \ge 0.
$$
As a consequence, Proposition~\ref{prop:H12n}-(3) gives that
$$
\forall t_0 >0, \, \forall\, n, \ell \ge 0, \quad
f \in L^\infty([t_0, \infty); H^n_{\ell} ).
$$

\medskip
\noindent
\textit{Step 3.}
Writing \eqref{eq:duhamel} starting from some time $t_* >0$ to be chosen later and using Theorem~\ref{thm:trou} (since $\Pi_0 h(t) = \Pi_0 Q(h(t),h(t)) = 0$ for any $t\ge 0$) it follows, for any $t \ge t_*$, that
$$
\bal
\| h(t) \|_{L^1(m)} 
&\le \| S_{\LL}(t - t_*)  h({t_*}) \|_{L^1(m)} + \int_{t_*}^t  \| S_{\LL}(t-\tau) Q(h(\tau), h(\tau))\|_{L^1(m)} \, d\tau \\
&\le C e^{-\lambda_0 t} \| h({t_*}) \|_{L^1(m)} + C\int_{t_*}^t e^{-\lambda_0 (t - \tau)} \|Q(h(\tau), h(\tau))\|_{L^1(m)} \, d\tau.
\eal
$$
From Lemma~\ref{lem:Q-Lpm} we have
$$
\| Q(h,h) \|_{L^1(m)} \lesssim 
\| h \|_{L^1_{\gamma+2}} \, \| \nabla^2 h \|_{L^1(\la v \ra^{\gamma+2} m)} 
+  \| h \|_{L^1} \, \| h \|_{L^1(m)} 
+ \| h \|_{L^{2}} \, \| h \|_{L^1(m)}.
$$
Moreover, we have the following interpolation inequality from \cite[Lemma B.1]{MiMo-cmp} 
$$
\| u \|_{H^n(m^{3/2})} \lesssim \| u \|_{L^1(m^3)}^{1/2} \, \| u \|_{H^{2n+1+3/2}}^{1/2}.
$$
Gathering the above bounds we get
$$
\bal
\| \nabla^2 h \|_{L^1(\la v \ra^{\gamma+2} m)} 
&\lesssim \| h \|_{H^2(\la v \ra^{\gamma+4} m)} \lesssim \| h \|_{H^2(m^{3/2})} \\
&\lesssim \| h \|_{L^1(m^3)}^{1/2} \, \| h \|_{H^{4 + 1 + 3/2}}^{1/2} \\
&\lesssim \| h \|_{L^1(m)}^{1/4} \, \| h \|_{L^1(m^5)}^{1/4} \, \| h \|_{H^{13/2}}^{1/2}, 
\eal
$$
where we have used H\"older's inequality in the last line. Moreover, using Nash's inequality we have
$$
\| h \|_{L^2} \lesssim \| h \|_{\dot H^{1}}^{3/5} \, \| h \|_{L^1}^{2/5}
\lesssim \| h \|_{\dot H^{1}}^{3/5} \, \| h \|_{L^1(m)}^{2/5}.
$$
Putting together the previous estimates it yields
$$
\bal
\| h(t) \|_{L^1(m)}
&\le C e^{-\lambda t} \| h(t_*) \|_{L^1(m)}
+ C \int_{t_*}^t e^{-\lambda_0 (t - \tau)}  \| h(\tau) \|_{H^{13/2}}^{1/2} \, \| h(\tau) \|_{L^1(m^5)}^{1/4} \, \| h(\tau) \|_{L^1(m)}^{1+1/4} \, d\tau \\
&\quad
+ C \int_{t_*}^t e^{-\lambda_0 (t - \tau)}   \| h(\tau) \|_{L^1(m)}^{2} \, d\tau \\
&\quad
+ C \int_{t_*}^t e^{-\lambda_0 (t - \tau)}  \| h(\tau) \|_{H^{1}}^{3/5} 
 \, \| h(\tau) \|_{L^1(m)}^{1+2/5} \, d\tau .
\eal
$$
Thanks to step 1, for any $\epsilon>0$ we can choose $t_* = t_*(\epsilon)$ such that
$$
 \sup_{t \ge t_*} \| h(t) \|_{L^1(m)} \le \sup_{t \ge t_*} \| h(t) \|_{L^1(m^5)} \le \epsilon.
$$
Also, from step 2 we get
$$
\sup_{t \ge t_*} \| h(t) \|_{H^{13/2}} \le C_1.
$$
Hence we obtain, for any $t \ge t_*$,
$$
\bal
\| h(t) \|_{L^1(m)}
\le C e^{-\lambda_0 t} \| h(t_*) \|_{L^1(m)}
+ C \left( C_1^{1/2} \epsilon^{1/4} + \epsilon^{3/4} + C_1^{3/5} \epsilon^{9/20}  \right)\int_{t_*}^t e^{-\lambda_0 (t - \tau)}  \| h(\tau) \|_{L^1(m)}^{1+1/4} \, d\tau .
\eal
$$
From this differential inequality, we argue as in \cite[Lemma 4.5]{Mouhot2} and choose $\epsilon >0$ small enough to obtain
$$
\forall\, t\ge t_*, \quad
\| h(t) \|_{L^1(m)} \le C \, e^{-\lambda_0 t} \| h(t_*) \|_{L^1(m)} \le C \, e^{-\lambda_0 t},
$$
from which, together with \eqref{eq:poly-conv} for $t < t_*$, we conclude the proof.
\end{proof}

\bigskip
\noindent{\bf Acknowledgements.}
The author would like to thank L.~Desvillettes, S.~Mischler and C.~Mouhot for fruitful discussions. The author is supported by the Fondation Math\'ematique Jacques Hadamard.



\begin{thebibliography}{10}

\bibitem{ALL}
R.~Alexandre, J.~Liao, and C.~Lin.
\newblock Some a priori estimates for the homogeneous {L}andau equation with
  soft potentials.
\newblock Preprint arxiv:1302.1814.

\bibitem{BM}
C.~Baranger and C.~Mouhot.
\newblock Explicit spectral gap estimates for the linearized {B}oltzmann and
  {L}andau operators with hard potentials.
\newblock {\em Rev.\ Matem.\ Iberoam.}, 21:819--841, 2005.

\bibitem{Beckner}
W.~Beckner.
\newblock Pitt's inequality with sharp convolution estimates.
\newblock {\em Proc.\ Amer.\ Math.\ Soc.}, 136(5):1817--1885, 2008.

\bibitem{KC4}
K.~Carrapatoso.
\newblock Exponential convergence to equilibrium for the homogeneous {L}andau
  equation with hard potentials.
\newblock \emph{Bull.\ Sci.\ math.} (2014), http://dx.doi.org/10.1016/j.bulsci.2014.12.002.

\bibitem{DL}
P.~Degond and M.~Lemou.
\newblock Dispersion relations for the linearized {F}okker-{P}lanck equation.
\newblock {\em Arch.\ Ration.\ Mech.\ Anal.}, 138:137--167, 1997.

\bibitem{DesMo}
L.~Desvillettes and C.~Mouhot.
\newblock Large time behavior for the a priori bounds for the solutions to the
  spatially homogeneous {B}oltzmann equation with soft potentials.
\newblock {\em Asymptotic Anal.}, 54:235--245, 2007.

\bibitem{DesVi1}
L.~Desvillettes and C.~Villani.
\newblock On the spatially homogeneous {L}andau equation for hard potentials.
  {I}. {E}xistence, uniqueness and smoothness.
\newblock {\em Comm.\ Partial Differential Equations}, 25(1-2):179--259, 2000.

\bibitem{DesVi2}
L.~Desvillettes and C.~Villani.
\newblock On the spatially homogeneous {L}andau equation for hard potentials.
  {II}. {$H$}-theorem and applications.
\newblock {\em Comm.\ Partial Differential Equations}, 25(1-2):261--298, 2000.

\bibitem{DesWen}
L.~Desvillettes and B.~Wennberg.
\newblock Smoothness of the solution of the spatially homogeneous {B}oltzmann
  equation without cutoff.
\newblock {\em Comm.\ Partial Differential Equations}, 29(1-2):133--155, 2004.


\bibitem{FG}
N.~Fournier and H.~Gu\'erin.
\newblock Well-posedness of the spatially homogeneous {L}andau equation for 
  soft potentials.
\newblock {\em J.\ Funct.\ Anal.}, 256(8):2542--2560, 2009.  


\bibitem{GMM}
M.~Gualdani, S.~Mischler, and C.~Mouhot.
\newblock Factorization for non-symmetric operators and exponential
  {H}-{T}heorem.
\newblock Preprint arxiv:1006.5523.

\bibitem{Guo}
Y.~Guo.
\newblock The {L}andau equation in a periodic box.
\newblock {\em Comm.\ Math.\ Phys.}, 231:391--434, 2002.

\bibitem{He}
L.~He.
\newblock Asymptotic analysis of the spatially homogeneous {B}oltzmman
  equation: grazing collisions limit.
\newblock {\em J.\ Stat.\ Phys.}, 155:151--210, 2014.

\bibitem{MiMo-FP}
S.~Mischler and C.~Mouhot.
\newblock Exponential stability of slowly decaying solutions to the kinetic
  {F}okker-{P}lanck equation.
\newblock In preparation.

\bibitem{MiMo-cmp}
S.~Mischler and C.~Mouhot.
\newblock Stability, convergence to self-similarity and elastic limit for the
  {B}oltzmann equation for inelastic hard spheres.
\newblock {\em Comm.\ Math.\ Phys.}, 288(2):431--502, 2009.

\bibitem{M}
C.~Mouhot.
\newblock Explicit coercivity estimates for the linearized {B}oltzmann and
  {L}andau operators.
\newblock {\em Comm.\ Part.\ Diff.\ Equations}, 261:1321--1348, 2006.

\bibitem{Mouhot2}
C.~Mouhot.
\newblock Rate of convergence to equilibrium for the spatially homogeneous
  {B}oltzmann equation with hard potentials.
\newblock {\em Comm.\ Math.\ Phys.}, 261:629--672, 2006.

\bibitem{MS}
C.~Mouhot and R.~Strain.
\newblock Spectral gap and coercivity estimates for the linearized {B}oltzmann
  collision operator without angular cutoff.
\newblock {\em J.\ Math.\ Pures Appl.}, 87:515--535, 2007.

\bibitem{TosVi-slow}
G.~Toscani and C.~Villani.
\newblock On the trend to equilibrium for some dissipative systems with slowly
  increasing a prior bounds.
\newblock {\em J.\ Statist.\ Phys.}, 98(5/6):1279--1309, 2000.

\bibitem{Tristani}
I.~Tristani.
\newblock Exponential convergence to equilibrium for the homogeneous
  {B}oltzmann equation without cut-off.
\newblock {\em J.\ Stat.\ Phys.}, 157(3):474--496, 2014.

\bibitem{Vi2}
C.~Villani.
\newblock On a new class of weak solutions to the spatially homogeneous
  {B}oltzmann and {L}andau equations.
\newblock {\em Arch.\ Rational Mech.\ Anal.}, 143(3):273--307, 1998.

\bibitem{Vi1}
C.~Villani.
\newblock On the spatially homogeneous {L}andau equation for {M}axwellian
  molecules.
\newblock {\em Math.\ Models Methods Appl.\ Sci.}, 8(6):957--983, 1998.

\bibitem{Villani-BoltzmannBook}
C.~Villani.
\newblock A review of mathematical topics in collisional kinetic theory.
\newblock In {\em Handbook of mathematical fluid dynamics, {V}ol. {I}}, pages
  71--305. North-Holland, Amsterdam, 2002.

\bibitem{Wu}
K.-C. Wu.
\newblock Global in time estimates for the spatially homogeneous {L}andau
  equation with soft potentials.
\newblock {\em J.\ Funct.\ Anal.}, 266:3134--3155, 2014.

\end{thebibliography}
\end{document}